\documentclass[12pt,a4paper]{article}
\usepackage{amsmath,amstext,amssymb,amscd, color}
\usepackage{graphicx}

\newcommand{\Xcomment}[1]{}

\newtheorem{theorem}{Theorem}[section]
\newtheorem{lemma}[theorem]{Lemma}

\newcommand{\SEC}[1]{\ref{sec:#1}}  
\newcommand{\SSEC}[1]{\ref{ssec:#1}}  

\makeatletter \@addtoreset{equation}{section} \makeatother

\newenvironment{proof}{\noindent{\bf Proof}~}%
{\hfill$\qed$\medskip}

\def\qed{ \ \vrule width.1cm height.3cm depth0cm}

\newenvironment{numitem1}{\refstepcounter{equation}\begin{enumerate}%
\item[(\thesection.\arabic{equation})]}{\end{enumerate}}

\newcommand{\refeq}[1]{(\ref{eq:#1})}  

 \makeatletter
\renewcommand{\section}{\@startsection{section}{1}{0pt}%
{-3.5ex plus -1ex minus -.2ex}{2.3ex plus .2ex}%
{\normalfont\Large}}
 \makeatother

  \Xcomment{
 \makeatletter
\renewcommand{\subsection}{\@startsection{subsection}{2}{0pt}%
{-3.0ex plus -1ex minus -.2ex}{-1.5ex plus .2ex}%
{\normalfont\large\bf}}
 \makeatother
 }

 \makeatletter
\renewcommand{\subsection}{\@startsection{subsection}{2}{0pt}%
{-3.0ex plus -1ex minus -.2ex}{1.5ex plus .2ex}%
{\normalfont\normalsize\bf}}
 \makeatother

\def\Rset{{\mathbb R}}
\def\Zset{{\mathbb Z}}
\def\Kset{{\mathbb K}}

\def\Ascr{{\cal A}}

\def\Cscr{{\cal C}}
\def\Dscr{{\cal D}}
\def\Escr{{\cal E}}
\def\Fscr{{\cal F}}

\def\Iscr{{\cal I}}
\def\Jscr{{\cal J}}

\def\Mscr{{\cal M}}

\def\Oscr{{\cal O}}
\def\Pscr{{\cal P}}

\def\Rscr{{\cal R}}
\def\Sscr{{\cal S}}
\def\Tscr{{\cal T}}

\def\Vscr{{\cal V}}
\def\Wscr{{\cal W}}

\def\frakU{\mathfrak{U}}
\def\frakL{\mathfrak{L}}

\def\tilde{\widetilde}
\def\hat{\widehat}
\def\bar{\overline}
\def\eps{\epsilon}
\def\vareps{\varepsilon}

\def\Zfr{Z^{\,\rm fr}}
\def\Zrear{Z^{\,\rm rear}}
\def\Zpfr{Z'^{\,\rm fr}}
\def\Zprear{Z'^{\,\rm rear}}
\def\Cfr{C^{\,\rm fr}}
\def\Cpfr{C'^{\,\rm fr}}
\def\Crear{C^{\,\rm rear}}

\def\Cprear{C'^{\,\rm rear}}

\def\Omegafr{\Omega^{\,\rm fr}}
\def\Omegarear{\Omega^{\,\rm rear}}

\def\Ufr{U^{\,\rm fr}}
\def\Urear{U^{\,\rm rear}}

\def\mmm{{[-m..m]}}
\def\mmmm{{[-m..m]^-}}
\def\mmmp{{[-m..m]^+}}

\def\Ndiam{N^{\lozenge}}

\def\Zup{Z^{\,\rm up}}
\def\Zlow{Z^{\,\rm low}}
\def\Kup{K^{\,\rm up}}
\def\Klow{K^{\,\rm low}}

\def\Qfrag{Q^\equiv}
\def\Cfrag{C^\equiv}
\def\Qpfrag{Q'^{\,\equiv}}

\begin{document}

 \title{The purity phenomenon for symmetric separated set-systems}

 \author{Vladimir I.~Danilov\thanks{Central Institute of Economics and
Mathematics of the RAS, 47, Nakhimovskii Prospect, 117418 Moscow, Russia;
email: danilov@cemi.rssi.ru.}
 \and
Alexander V.~Karzanov
\thanks{Central Institute of Economics and Mathematics of
the RAS, 47, Nakhimovskii Prospect, 117418 Moscow, Russia; email:
akarzanov7@gmail.com. Corresponding author.
}
  \and
Gleb A.~Koshevoy
\thanks{The Institute for Information Transmission Problems of
the RAS, 19, Bol'shoi Karetnyi per., 127051 Moscow, Russia; email:
koshevoyga@gmail.com. Supported in part by grant RSF 21-11-00283. }
  }

\date{}

 \maketitle

\begin{abstract}
Let $n$ be a positive integer. A collection $\Sscr$ of subsets of
$[n]=\{1,\ldots,n\}$ is called \emph{symmetric} if $X\in \Sscr$ implies
$X^\ast\in\Sscr$, where $X^\ast:=\{i\in [n]\colon n-i+1\notin X\}$. As the main results of this paper, one shows
that in each of the three types of separation relations: \emph{strong},
\emph{weak} and \emph{chord} ones, the following ``purity phenomenon'' takes
place: all inclusion-wise maximal symmetric separated collections in $2^{[n]}$
have the same cardinality. These give ``symmetric versions'' of well-known
results on the purity of usual strongly, weakly and chord separated collections
of subsets of $[n]$, and in the case of weak separation, this extends a
result due to Karpman on the purity of symmetric weakly separated collections
in $\binom{[n]}{n/2}$ for $n$ even.
 \end{abstract}

\smallskip


\baselineskip=15pt
\parskip=1pt

\section{Introduction}  \label{sec:intr}

Let $n$ be a positive integer and let $[n]$ denote the set $\{1,2,\ldots,n\}$.
In this paper, we deal with three known types of relations on subsets of $[n]$,
called strong, weak and chord separation ones. In the definitions below, for
subsets $A,B\subseteq [n]$, we write $A<B$ if the maximal element $\max(A)$ of
$A$ is smaller than the minimal element $\min(B)$ of $B$ (letting
$\max(\emptyset):=-\infty$ and $\min(\emptyset):=\infty$). When $B-A\ne\emptyset$, we say that $A$
\emph{surrounds} $B$ if $\min(A-B)<\min(B-A)$ and $\max(A-B)>\max(B-A)$ (where
$A-B$ denotes the set difference $\{i\colon A\ni i\notin B\}$).
\medskip

\noindent \textbf{Definitions.} Sets $A,B\subseteq[n]$ are called
\emph{strongly separated} (from each other) if there are no three elements
$i<j<k$ of $[n]$ such that one of $A-B$ and $B-A$ contains $i,k$, and the other
contains $j$ (equivalently, either $A-B<B-A$ or $B-A<A-B$ or $A=B$). Sets
$A,B\subseteq[n]$ are called \emph{chord separated} if there are no four
elements $i<j<k<\ell$ of $[n]$ such that one of $A-B$ and $B-A$ contains $i,k$,
and the other contains $j,\ell$. Chord separated sets $A,B\subseteq[n]$ are
called \emph{weakly separated} if the following additional condition holds: if
$A$ surrounds $B$ then $|A|\le|B|$, and if $B$ surrounds $A$ then $|B|\le|A|$
(where $|A'|$ is the number of elements in $A'$). Accordingly, a collection
$\Ascr\subseteq 2^{[n]}$ of subsets of $[n]$ is called strongly (weakly, chord)
separated if any two members of $\Ascr$ are strongly (resp. weakly, chord)
separated.
 \medskip

The notions of strong and weak separations were introduced by Leclerc and
Zelevinsky~\cite{LZ}, and the notion of chord separation by
Galashin~\cite{gal}. The first two notions appeared in~\cite{LZ} in
connection with the problem of characterizing quasi-commuting flag minors of a
quantum matrix. (In particular, one shows there that in the quantized
coordinate ring $\Oscr_q(\Mscr_{m,n}(\Kset))$ of $m\times n$ matrices over a field $\Kset$, where $q\in\Kset^\ast$, two flag minors $[I]$ and $[J]$ quasi-commute, i.e., satisfy the equality $[J][I]=q^c[I][J]$ for some integer $c$,  if and only if their column sets $I,J\subseteq[n]$ are weakly separated.)  For a discussion on this and wider relations between the weak/strong separation and quantum minors, see also~\cite[Sect.~8]{DK}).

For brevity we will refer to strongly, weakly, and chord
separated collections as \emph{s-}, \emph{w-}, and \emph{c-collections},
respectively. The sets of such collections in $2^{[n]}$ are denoted by ${\bf
S}_n$, ${\bf W}_n$, and ${\bf C}_n$, respectively. As is shown in~\cite{LZ},
  \begin{numitem1} \label{eq:sn}
the maximal possible sizes of strongly and weakly separated collections in
$2^{[n]}$ are the same and equal to
$s_n:=\binom{n}{2}+\binom{n}{1}+\binom{n}{0}$ ($=\frac12 n(n+1)+1$).
  \end{numitem1}

When dealing with one or another set (class) $\bf C$ of collections in
$2^{[n]}$, one says that  $\bf C$ is \emph{pure} if any maximal by inclusion
collection in it is maximal by size (viz. number of members). Leclerc and
Zelevinsky showed in~\cite{LZ} that the set ${\bf S}_n$ is pure and conjectured
that ${\bf W}_n$ is pure as well. This was affirmatively answered
in~\cite{DKK1}, by proving that
\begin{numitem1} \label{eq:LZ}
any w-collection in $2^{[n]}$ can be extended to a w-collection of size $s_n$.
  \end{numitem1}
(One application of the purity of ${\bf W}_n$ mentioned in~\cite{LZ} concerns the dual canonical basis in $\Oscr_q(\Mscr_{m,n}(\Kset))$ containing all quasi-commuting monomials.) For other interesting classes of w-collections with the purity behavior, see~\cite{OPS,OS,DKK2}.) The corresponding purity result for chord separation was obtained by Galashin~\cite{gal}:

\begin{numitem1} \label{eq:gal}
any c-collection in $2^{[n]}$ can be extended to a c-collection of size
$c_n:=\binom{n}{3}+\binom{n}{2}+\binom{n}{1}+\binom{n}{0}$.
  \end{numitem1}

Recently Karpman~\cite{karp} revealed the purity for a special class of
symmetric w-collections, and we now outline this result.

It will be convenient to us to interpret the elements of the ground set $[n]$
as \emph{colors}. Consider the following relations on $[n]$ and $2^{[n]}$:
  \begin{numitem1} \label{eq:implicat}
  \begin{itemize}
\item[(i)]
for $i\in [n]$, define $i^\circ:=n+1-i$ (the ``complementary color'' to $i$);
\item[(ii)] for $A\subseteq [n]$, define $\bar A:=[n]-A$ (the ``complementary set'' to $A$);
\item[(iii)] for $A\subseteq[n]$, define $A^\ast:=\{i^\circ\colon i\in \bar A\}$.
  \end{itemize}
  \end{numitem1}

Clearly (i) and (ii) are involutions: $(i^\circ)^\circ=i$ and $\bar{\bar A}=A$.
And (iii) is viewed as the composition of these two involutions (which commute); so it is an
involution as well: $(A^\ast)^\ast=A$. Involution~(iii), which is of most
interest to us in this paper, was introduced by Karpman~\cite{karp} in the
special case when $n$ is even and $|A|=n/2$. (In that case, the author treats the so-called symmetric plabic graphs, which are closely related to the corresponding Lagrangian Grassmannian, the space of maximal isotropic subspaces with respect to a symplectic form.)

For convenience, we will refer to
involution~(iii) as the \emph{K-involution} and use the following definition
for corresponding set-systems.
 \medskip
 
\noindent\textbf{Definition.} A collection $\Sscr\subseteq 2^{[n]}$ is called
\emph{symmetric} if it is closed under the K-involution, i.e., $A\in\Sscr$
implies $A^\ast\in\Sscr$.
  \medskip

Using a technique of plabic tilings and relying on the purity of the set of
w-collections in a ``discrete Grassmannian'' $\binom{[n]}{m}=\{A\subseteq[n] \colon |A|=m\}$ with $m\in[n]$ (cf.~\cite{OPS}), Karpman showed the following

\begin{theorem}[\cite{karp}] \label{tm:karp}
For $n$ even, all inclusion-wise maximal symmetric w-collections in
$\binom{[n]}{n/2}$ have the same cardinality, which is equal to $n^2/4+1$.
  \end{theorem}

\noindent(This coincides with the maximum cardinality when the symmetry
condition is discarded.) We give the following generalization of that result.
  \begin{theorem} \label{tm:symm-ws}
For any $n\in\Zset_{>0}$, all inclusion-wise maximal symmetric w-collections in
$2^{[n]}$ have the same cardinality. When $n$ is even, it is equal to $s_n$.
When $n$ is odd, it is equal to $s_n-(n-1)/2$.
  \end{theorem}

\noindent\textbf{Remark 1.} Theorem~\ref{tm:symm-ws} implies
Theorem~\ref{tm:karp}. Moreover, one can show a sharper result, as follows. For
an even $n$ and an integer $k$ such that $0\le k<n/2$, let $\Lambda_{n,k}$
denote the union of sets $\binom{[n]}{i}$ over $n/2-k\le i\le n/2+k$. We assert
that: $(\ast)$ all inclusion-wise maximal symmetric w-collections $\Sscr$ in
$\Lambda_{n,k}$ have the same size (which gives Theorem~\ref{tm:karp} when
$k=0$). To see this, given a symmetric w-collection $\Sscr$ in $\Lambda_{n,k}$,
extend it to the collection $\Wscr\subset 2^{[n]}$ by adding the set $\Iscr$ of
all intervals $I\subseteq[n]$ of size $\ge n/2+k$ and the set $\Jscr$ of all
co-intervals $J\subset[n]$ of size $\le n/2-k$ (including the empty co-interval
$\emptyset$). Hereinafter, an \emph{interval} in $[n]$ is meant to be a set of
the form $\{a,a+1,\ldots,b\}\subseteq[n]$, denoted as $[a..b]$ (in particular,
$[n]=[1..n]$), and a \emph{co-interval} is the complement $\bar I=[n]-I$ of an
interval $I$. One can check that: (i) each interval $I$ is weakly separated
from any set $A\subseteq [n]$ with $|A|\le|I|$; symmetrically, each co-interval
$J$ is weakly separated from any $B\subseteq [n]$ with $|B|\ge |J|$; (ii) for
any $I\in \Iscr$, \;$I^\ast$ is a co-interval in $\Jscr$, a vice versa; (iii)
any set $A\in 2^{[n]}-\Iscr$ with $|A|>n/2+k$ is not weakly separated from some
$I\in\Iscr$; symmetrically, any $B\in 2^{[n]}-\Jscr$ with $|B|< n/2-k$ is not
weakly separated from some $J\in\Jscr$. These properties imply that $\Sscr$ as
above is inclusion-wise maximal in $\Lambda_{n,k}$ if and only if
$\Wscr:=\Sscr\cup\Iscr\cup\Jscr$ is inclusion-wise maximal in $2^{[n]}$, whence
assertion $(\ast)$ follows from Theorem~\ref{tm:symm-ws}. Also, given $n,k$,
we can easily express $|\Wscr|$ via $|\Sscr|$, and back.
  \medskip

The proof of Theorem~\ref{tm:symm-ws} relies on~\refeq{LZ} and is based on a
geometric approach: it attracts a machinery of \emph{combined tilings}, or
\emph{combies} for short, which are certain planar polyhedral complexes on
two-dimensional \emph{zonogons} introduced and studied in~\cite{DKK2}. It turns
out that when $n$ is even, there is a natural bijection between the maximal
symmetric w-collections in $2^{[n]}$ and the so-called ``symmetric combies'' on
the zonogon $Z(n,2)$ (this is analogous to the existence of a bijection between
usual maximal w-collections and combies, see~\cite[Theorems~3.4,3.5]{DKK2}).

As a by-product of our method of proof of Theorem~\ref{tm:symm-ws}, we obtain a
similar purity result for the strong separation:

  \begin{numitem1} \label{eq:symm-str}
all inclusion-wise maximal symmetric s-collections in $2^{[n]}$ have the same
cardinality; it is equal to $s_n$ when $n$ is even, and $s_n-(n-1)/2$ when $n$
is odd.
  \end{numitem1}

The next group of  results of this paper concerns symmetric chord
separated collections. Our main theorem in this direction is as follows.

\begin{theorem} \label{tm:symm-chord}
For any $n\in\Zset_{>0}$, all inclusion-wise maximal symmetric c-collections in
$2^{[n]}$ have the same cardinality $c_n$.
  \end{theorem}
(Note that this can be regarded as a generalization of Theorem~\ref{tm:karp} as
well since within any domain of the form $\binom{[n]}{k}$ the notions of
weak and chord separations coincide.)

An important ingredient of the proof of this theorem is the geometric
characterization of maximal chord separated collections in $2^{[n]}$ in terms
of \emph{fine zonotopal tilings}, or \emph{cubillages}, of 3-dimensional cyclic
zonotopes $Z(n,3)$, due to Galashin~\cite{gal} (the term ``cubillage'' that we prefer to use in this paper appeared
in~\cite{KV}). We establish a symmetric analog of that nice property (valid for
both even and odd cases of $n$): any maximal symmetric c-collection in
$2^{[n]}$ can be expressed by the vertex set of a \emph{symmetric cubillage} on
$Z(n,3)$.
\medskip

This paper is organized as follows. Section~\SEC{prelim} contains basic
definitions and reviews some known facts. In particular, it explains the
notions of combined tilings, or \emph{combies}, and fine zonotopal tilings, or
\emph{cubillages} (in the 3-dimensional case), and recalls basic results on them
needed to us. Section~\SEC{symm-w-even} deals with the even color case of
symmetric weakly separated collections and proves Theorem~\ref{tm:symm-ws} for
$n$ even. The odd color case of this theorem is studied in
Section~\SEC{symm-w-odd}. Section~\SEC{symm-c-even} is devoted to the even color case of symmetric chord separated collections, giving the proof of
Theorem~\ref{tm:symm-chord} for $n$ even. The odd case of this theorem is shown
in the concluding Section~\SEC{symm-c-odd}. This section finishes with a slightly
sharper version of Theorem~\ref{tm:symm-chord} (in Remark~6). Also we add two more results in the ends of Sections~\SEC{symm-c-even} and~\SEC{symm-c-odd} (Theorems~\ref{tm:symm-w-c} and~\ref{tm:symm-w-c-odd}) which are devoted to geometric constructions related to embeddings of maximal symmetric w-collections in maximal c-collections.)
  \medskip

It should be noted that the above purity results do not remain true for
``higher'' symmetric separation. Recall that sets $A,B\subseteq[n]$ are called
(strongly) $k$-separated if there are no $k+2$ elements
$i_1<i_2<\cdots<i_{k+2}$ of $[n]$ such that the elements with odd indexes
belong to one, while those with even indexes to the other set among $A-B$ and
$B-A$. In particular, chord separated sets are just 2-separated ones. As is
shown in~\cite{GP}, when $k\ge 3$, a maximal by inclusion $k$-separated
collection in $2^{[n]}$ need not be maximal by size. (In fact, a counterexample
to the purity with $k=3$ given there can be adjusted to the symmetric
3-separation as well.)

Surprisingly, the \emph{maximal by size symmetric} $k$-separated  collections in $2^{[n]}$ possess nice structural and geometric properties; they are systematically studied in~\cite{DKK5} in the context of higher Bruhat orders of types B and C (where some open questions and conjectures are raised as well). One important property among those is that such collections for $n,k$ even can be connected by use of symmetric local mutations (or ``flips'') yielding a poset structure with one minimal and one maximal elements. More about symmetric flips will appear in a forthcoming paper.


\section{Preliminaries}  \label{sec:prelim}

In this section we give additional definitions and notation and review some
facts about combined tilings and cubillages needed for the proofs of
Theorems~\ref{tm:symm-ws} and~\ref{tm:symm-chord}.
 \smallskip

\noindent$\bullet$ ~For an edge $e$ of a directed graph $G$ without parallel
edges, we write $e=(u,v)$ if $e$ connects vertices $u$ and $v$ and is directed
(or ``going'') from $u$ to $v$. A path in $G$ is a sequence
$P=(v_0,e_1,v_1,\ldots,e_k,v_k)$ in which each $e_i$ is an edge connecting
vertices $v_{i-1}$ and $v_i$. It is called a \emph{directed} path if each edge
$e_i$ is directed from $v_{i-1}$ to $v_i$. When it is not confusing, we may
write $P=v_0v_1\ldots v_k$ (using notation via vertices).
\smallskip

\noindent$\bullet$ ~Let $n$ be a positive integer. Define $m:=\lfloor
n/2\rfloor$; then $n=2m$ if $n$ is even, and $n=2m+1$ if $n$ is odd. Instead of
\emph{colors} $1,2,\ldots, n$ (forming $[n]$), it will be often more convenient
to deal with the set of ``symmetric colors'' $-i,i$ for $i=1,\ldots,m$, to
which we also add color 0 when $n$ is odd. This gives the \emph{symmetrized}
color sets
  \begin{numitem1} \label{eq:symm-colors}
$\{-m,\ldots,-1,0,1,\ldots,m\}$ denoted as $[-m..m]$ when $n$ is odd, and
$\{-m,\ldots,-1,1,\ldots,m\}$ denoted as $[-m..m]^-$ when $n$ is even.
  \end{numitem1}

So $i^\circ=-i$ for each color $i$, and the only self-symmetric color is 0
(when $n$ is odd).
\smallskip

\noindent$\bullet$ ~For $A\subseteq[n]$ and $p=0,1,2$, define $\Pi_p(A)$ to be
the set of symmetric color pairs $\{i,i^\circ\}$ in $[n]$ such that
$|A\cap\{i,i^\circ\}|=p$. We say that the pairs in $\Pi_0(A)$, $\Pi_1(A)$, and
$\Pi_2(A)$ are, respectively, \emph{poor}, \emph{ordinary}, and \emph{full} for
$A$. (Note that if $n$ is odd and $i=m+1$, then the ``middle'' pair
$\{i,i^\circ=i\}$ is regarded as either poor or full in $A$.) In these terms,
we observe a useful relationship between symmetric sets $A$ and $A^\ast$:
   \begin{numitem1} \label{eq:pi012}
$\Pi_0(A)=\Pi_2(A^\ast)$, $\Pi_2(A)=\Pi_0(A^\ast)$, and
$\Pi_1(A)=\Pi_1(A^\ast)$; moreover, the ordinary pairs are \emph{stable}, in
the sense that if $i\in A\not\ni i^\circ$ then $i\in A^\ast\not\ni i^\circ$.
  \end{numitem1}
Indeed, $i\in A\not\ni i^\circ$ implies $i\notin \bar A\ni i^\circ$, whence
$i\in A^\ast\not\ni i^\circ$. And $i,i^\circ\notin A$ implies $i,i^\circ\in
\bar A$, whence $i,i^\circ\in A^\ast$.

When dealing with colors in the symmetrized form as above, the sets $\Pi_p(A)$
are defined accordingly.

For a symmetric collection $\Sscr\subseteq 2^{[n]}$ and $h=0,1,\ldots,n$,
define $h$-th \emph{level} of $\Sscr$ as
  $$
  \Sscr_h:=\{A\in\Sscr\colon |A|=h\}.
  $$
Then $\Sscr_h$ consists of the sets $A\in\Sscr$ with
$\Pi_1(A)+\frac12\Pi_2(A)=h$, and~\refeq{pi012} implies
  $$
  \Sscr_h=(\Sscr_{n-h})^\ast,
  $$
where we extend the operator $\ast$ to collections in $2^{[n]}$ in a natural way.
 \smallskip

The next two subsections review the constructions of combined tilings and
cubillages, which are the key objects in our proofs of
Theorems~\ref{tm:symm-ws} and~\ref{tm:symm-chord}, respectively.

\subsection{Zonogon and combies.} \label{ssec:combi}

Let $\Xi$ be a set of $n$ vectors $\xi_i=(x_i,y_i)\in \Rset^2$ such that
  \begin{numitem1} \label{eq:xi-x-y}
  $x_1<\cdots <x_n$ and $y_i=1-\delta_i$, $i=1,\ldots,n$,
    \end{numitem1}
where each $\delta_i$ is a sufficiently small positive real. In addition, we
assume that
  \begin{numitem1} \label{eq:additional}
  \begin{itemize}
\item[(i)] $\Xi$ satisfies the \emph{strict concavity} condition:  for any $i<j<k$,
there exist $\lambda,\lambda'\in \Rset_{>0}$ such that $\lambda+\lambda'>1$ and
$\xi_j=\lambda\xi_i+\lambda'\xi_k$; and
 \item[(ii)]
the vectors in $\Xi$ are $\Zset_2$-independent, i.e., all 0,1-combinations of
these vectors are  different.
 \end{itemize}
  \end{numitem1}

An example is illustrated in the picture; here $n=5$, $(x_1,\ldots,x_5)=(-2,-1,0,1,2)$ and $y_i=1-x_i^2/12$.

\vspace{-0cm}
\begin{center}
\includegraphics{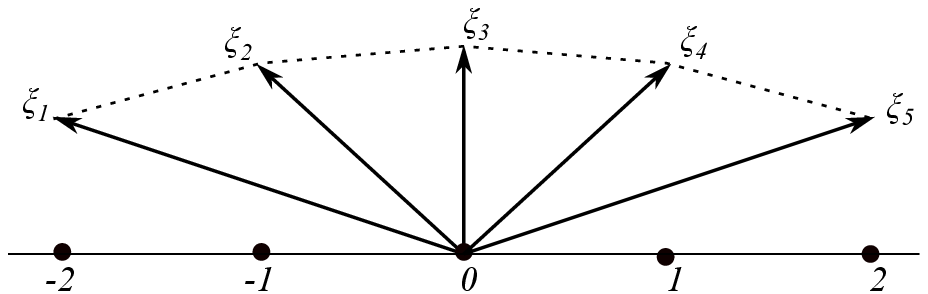}
\end{center}
\vspace{-0cm}

The \emph{zonogon} generated by $\Xi$ is the $2n$-gon being the Minkowski sum
of segments $[0,\xi_i]$, $i=1,\ldots,n$, i.e., the set
  $$
Z=Z(\Xi):=\{\lambda_1\xi_1+\ldots+ \lambda_n\xi_n\colon \lambda_i\in\Rset,\;
0\le\lambda_i\le 1,\; i=1,\ldots,n\}.
  $$
When the choice of $\Xi$ is not important to us (subject
to~\refeq{xi-x-y},\refeq{additional}), we may denote $Z$ as $Z(n,2)$. Each
subset $X\subseteq[n]$ is identified with the point $\sum_{i\in X} \xi_i$ in
$Z$ (due to~\refeq{additional}(ii), different subsets are identified with
different points).

Besides $\xi_1,\ldots,\xi_n$, we use the vectors $\eps_{ij}:=\xi_j-\xi_i$ for
$1\le i<j\le n$.

A \emph{combined tiling}, or a \emph{combi} for short, is a subdivision $K$ of
$Z=Z(\Xi)$ into convex polygons specified below and called \emph{tiles}. Any
two intersecting tiles share a common vertex or edge, and each edge of the
boundary of $Z$ belongs to exactly one tile.

We associate to $K$ the planar graph $(V_K,E_K)$ whose vertex set $V_K$ and
edge set $E_K$ are formed by the vertices and edges occurring in tiles. Each
vertex is (a point identified with) a subset of $[n]$. And each edge is a line
segment viewed as a parallel transfer of either $\xi_i$ or $\eps_{ij}$ for some
$i<j$. In the former case, it is called an edge of \emph{type} or \emph{color}
$i$, or an $i$-\emph{edge}, and in the latter case, an edge of \emph{type}
$ij$, or an $ij$-\emph{edge}. An $i$-edge ($ij$-edge) is directed according to
the direction of $\xi_i$ (resp. $\eps_{ij}$), and $(V_K,E_K)$ is the
corresponding directed graph. In particular, the \emph{left boundary} of $K$
(and of $Z$) is the directed path $v_0v_1\ldots v_n$ in which each vertex $v_i$
represents the interval $[i]$ (and the edge from $v_{i-1}$ to $v_i$ has color
$i$). And the \emph{right boundary} is the directed path $v'_0v'_1\ldots v'_n$
in which $v'_i$ represents the interval $[n+1-i..n]$.

In what follows, for disjoint subsets $A$ and $\{a,\ldots,b\}$ of $[n]$, we
will use the abbreviated notation $Aa\ldots b$ for $A\cup\{a,\ldots,b\}$, and
write $A-c$ for $A-\{c\}$ when $c\in A$.

There are three sorts of tiles in a combi $K$: $\Delta$-tiles, $\nabla$-tiles,
and lenses.
 \smallskip

I. A $\Delta$-\emph{tile} ($\nabla$-\emph{tile}) is a triangle with vertices
$A,B,C\subseteq[n]$ and edges $(B,A),(C,A),(B,C)$ (resp. $(A,C),(A,B),(B,C)$)
of types $j$, $i$ and $ij$, respectively, where $i<j$. We denote this tile as
$\Delta(A|BC)$ (resp. $\nabla(A|BC)$). See the left and middle fragments of the
picture.

\vspace{-0.2cm}
\begin{center}
\includegraphics{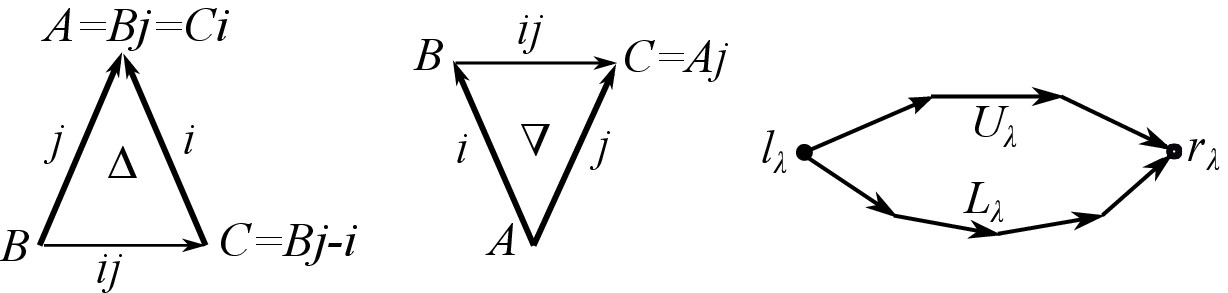}
\end{center}
\vspace{-0.2cm}

II. In a \emph{lens} $\lambda$, the boundary is formed by two directed paths
$U_\lambda$ and $L_\lambda$, with at least two edges in each, having the same
beginning vertex $\ell_\lambda$ and the same end vertex $r_\lambda$; see the
right fragment of the above picture. The \emph{upper boundary}
$U_\lambda=(v_0,e_1,v_1,\ldots,e_p,v_p)$ is such that $v_0=\ell_\lambda$,
$v_p=r_\lambda$, and $v_k=Xi_k$ for $k=0,\ldots,p$, where $p\ge 2$,
$X\subset[n]$ and $i_0<i_1<\cdots <i_p$ (so $k$-th edge $e_k$ is of type
$i_{k-1}i_k$). And the \emph{lower boundary}
$L_\lambda=(u_0,e'_1,u_1,\ldots,e'_q,u_q)$ is such that $u_0=\ell_\lambda$,
$u_q=r_\lambda$, and $u_m=Y-j_m$ for $m=0,\ldots,q$, where $q\ge 2$,
$Y\subseteq [n]$ and $j_0>j_1>\cdots>j_q$ (so $m$-th edge $e'_m$ is of type
$j_mj_{m-1}$). Then $Y=Xi_0j_0=Xi_pj_q$, implying $i_0=j_q$ and $i_p=j_0$. Note
that $X$ as well as $Y$ need not be a vertex in $K$. Due to the concavity
condition~\refeq{additional}(i), $\lambda$ is a convex polygon of which
vertices are exactly the vertices of $U_\lambda\cup L_\lambda$.
  \smallskip

Besides, we will deal with two derivatives of combies (cf.~\cite[Sect.~6.3]{DKK3}).
\smallskip

A. A \emph{quasi-combi} $K$ differs from a combi by the
condition that in each lens $\lambda$, either the upper boundary $U_\lambda$ or
the lower boundary $L_\lambda$ (not both) can consist of only one edge; we
refer to $\lambda$ as a \emph{lower semi-lens} in the former case, and as an
\emph{upper semi-lens} in the latter case. If, in addition, no two upper semi-lenses can share an edge, and similarly for the lower semi-lenses, then we say that a quasi-combi $K$ is \emph{fine}. Typically, a fine quasi-combi is produced from a
combi by subdividing each lens $\lambda$ of the latter into two semi-lenses of
different types (lower and upper ones) along the segment
$[\ell_\lambda,r_\lambda]$. See the left fragment of the picture.

\vspace{-0.1cm}
\begin{center}
\includegraphics[scale=1.2]{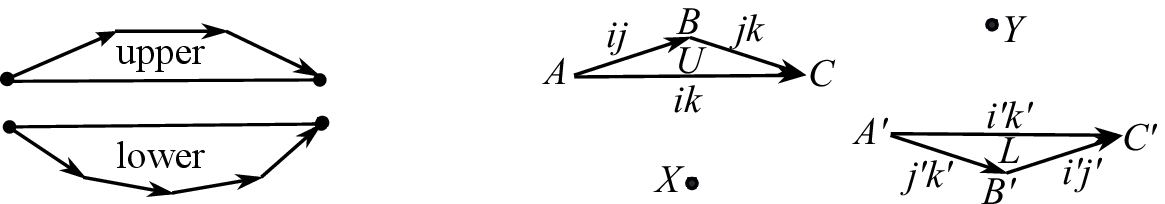}
\end{center}
\vspace{-0.1cm}

B. The second derivative is of most use in this paper. It was introduced
in~\cite{DKK3} under the name of a \emph{fully triangulated quasi-combi}, that
we will abbreviate as an \emph{ftq-combi}. In an ftq-combi $K$, all lenses are semi-lenses and, moreover, they are triangles. So an \emph{upper semi-lens} is a triangle
$U=U(ABC)$ formed by three vertices $A,B,C$ and three directed edges of types
$ij,jk,ik$ such that $i<j<k$. The vertices are expressed as $A=Xi$, $B=Xj$ and
$C=Xk$ for some $X\subset[n]$, called the \emph{root} of $U$ (which is not
necessarily a vertex of $K$). And a \emph{lower semi-lens} is a triangle
$L=L(A'B'C')$ formed by vertices $A',B',C'$ and directed edges of types
$j'k',i'j',i'k'$ such that $i'<j'<k'$. The vertices are viewed as $A'=Y-k'$,
$B'=Y-j'$ and $C'=Y-i'$ for some $Y\subseteq[n]$, called the \emph{root} of
$L$. See the middle and right fragments of the above picture. Typically, an ftq-combi is produced from a combi by subdividing each lens into one lower and one upper
semi-lenses (forming a fine quasi-combi as in A) and then subdividing the former
into lower triangles, and the latter into upper ones. Conversely, starting from an ftq-combi, if we choose, step by step, a pair of semi-lenses that share an edge and have the same type and replace them by their union, then we eventually obtain a fine quasi-combi (which preserves the set of vertices and does not depend on the choice of pairs in the process).
\smallskip

The picture below illustrates fragments of a combi (left) and an ftq-combi
(right); here lenses ($\lambda$ and $\lambda'$) and triangular semi-lenses are
drawn bold.

 \vspace{-0.2cm}
\begin{center}
\includegraphics[scale=0.8]{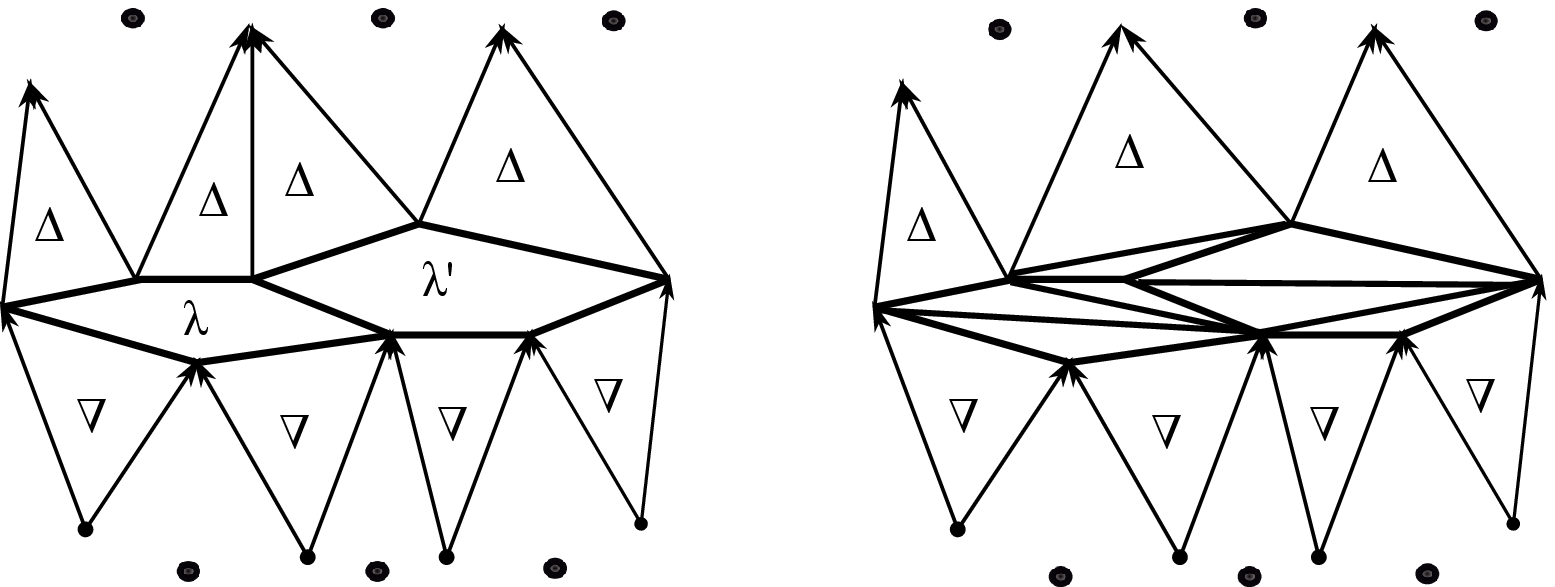}
\end{center}
\vspace{-0.2cm}

\noindent\textbf{Remark 2.} In the definition of a combi given in~\cite{DKK2},
the generators $\xi_i$ are assumed to have equal euclidean lengths. However,
taking generators subject to~\refeq{xi-x-y} does not affect, in essence, the
structure of combies, as well as results on them, and we may vary generators,
with a due care, when needed. To simplify visualizations, it is convenient to
think of edges of type $i$ as ``almost vertical'', while of those of type $ij$
as ``almost horizontal'' (since the values $\delta_i$ in~\refeq{xi-x-y} are
small). Note that any rhombus tiling turns into a combi without lenses in a
natural way: each rhombus is subdivided into two ``semi-rhombi'' $\Delta$ and
$\nabla$ by drawing the ``almost horizontal'' diagonal in it. Note that from
axioms~\refeq{xi-x-y},\refeq{additional}(i) it follows that
  \begin{numitem1} \label{eq:all_differ}
all vectors $\xi_1,\ldots,\xi_n$ and $\eps_{ij}$, $1\le i<j\le n$, are
different.
  \end{numitem1}
In particular, this implies that if some $\Delta$-tile and $\nabla$-tile share
an ``almost horizontal'' edge (i.e. they are of the form $\Delta(A|BC)$ and $\nabla(A'|B'C')$ with
$BC=B'C'$) then their union is a parallelogram. As a consequence, any ftq-combi
without semi-lenses is equivalent to a rhombus tiling (and vice versa).
 \medskip

The central result on combies shown in~\cite{DKK2} (which in turn relies on the
purity of ${\bf W}_n$ shown in~\cite{DKK1}) is that there is a one-to-one
correspondence between the set of combies $K$ on $Z(n,2)$ and the set ${\bf
W}_n$ of maximal w-collections $\Wscr$ in $2^{[n]}$; it is given by $K\mapsto
V_K=:\Wscr$. As a consequence,
  \begin{numitem1} \label{eq:ftq-combi}
for any ftq-combi $K$ on $Z(n,2)$, the set $V_K$ of vertices (regarded as
subsets of $[n]$) forms a maximal w-collection in $2^{[n]}$, and conversely,
any w-collection in $2^{[n]}$ is representable by the vertex set $V_K$ of some
ftq-combi $K$.
  \end{numitem1}

Next, in order to handle symmetric ws-collections, it is convenient to assume
that the set $\Xi$ of generating vectors $\xi_i=(x_i,y_i)$ is symmetric, in the
sense that:
  \begin{numitem1} \label{eq:symgen}
$x_i=-x_{i^\circ}$ and $y_i=y_{i^\circ}$ for each $i\in[n]$
  \end{numitem1}
cf.~\refeq{xi-x-y}. Note that in this case we may assume that
conditions~\refeq{additional}(i),(ii) continue to hold. (To provide this, we
first assign $Z_2$-independent numbers $x_i$ for $i=1,\ldots,m=\lfloor
n/2\rfloor$ so that $x_1<\cdots <x_m<0$, and accordingly define
$x_{m+1},\ldots,x_n$ by symmetry (where $x_{m+1}=0$ if $n$ is odd). Then assign
symmetric $y_1,\ldots,y_n$ (with $y_i=y_{i^\circ}$) so as to satisfy the
concavity condition~\refeq{additional}(i). Then slightly perturbing the values
$y_i$, if needed, we ensure that all 0,1,2-combinations of the numbers
$y_1,\ldots,y_m$ are different. One can see that the resulting vectors
$\xi_1,\ldots,\xi_n$ are $Z_2$-independent, yielding~\refeq{additional}(ii).)

When $n$ is even, \refeq{symgen} implies that the zonogon $Z:=Z(\Xi)$ admits
the reflection with respect to the horizontal line
  \begin{equation} \label{eq:M}
  M:=\{(x,y)\in Z\colon y=y_1+\cdots+y_{n/2}\},
  \end{equation}
called the \emph{middle line} of $Z$. Moreover, we observe that
  \begin{numitem1} \label{eq:mirrorA}
for any $A\subseteq[n]$, the sets $A$ and $A^\ast$ are symmetric w.r.t. $M$, or
$M$-\emph{symmetric} for short, which means that their corresponding points
$(x_A,y_A)$ and $(x_{A^\ast},y_{A^\ast})$ in $Z$ satisfy $x_A=x_{A^\ast}$ and
$y_A-y^M=y^M-y_{A^\ast}$, where $y^M:=y_1+\cdots+y_{n/2}$.
 \end{numitem1}

\noindent(In particular, $\emptyset$ is $M$-symmetric to $[n]$, and
$\{i,i^\circ\}$ is $M$-symmetric to $[n]-\{i,i^\circ\}$ for each $i\in[n]$).
Clearly if $A\subset[n]$ lies on $M$, then $|A|=n/2$, implying that the amounts
of poor and full pairs in $A$ are equal: $|\Pi_0(A)|=|\Pi_2(A)|$. Moreover, in
view of~\refeq{pi012} and~\refeq{additional}(ii), for any $A\subset[n]$, the
points $(x_A,y_A)$ and $(x_{A^\ast},y_{A^\ast})$ coincide if and only if
$\Pi_0(A)=\Pi_2(A)=\emptyset$. This gives the useful property that
  \begin{numitem1} \label{eq:sets-in-M}
all sets $A\subset[n]$ with $\Pi_0(A)=\Pi_2(A)=\emptyset$ and only these are
contained in $M$.
  \end{numitem1}
In other words, $M$ contains merely \emph{self-symmetric} sets $A=A^\ast$ and all these.

Also the middle line $M$ enables us to define an important class of ftq-combies
(extending the notion of $M$-symmetry to subsets of points in $Z$ in a natural
way).
 \medskip

\noindent\textbf{Definition.} Let $n$ be even. An ftq-combi $K$ on $Z$ is
called \emph{symmetric} if for any tile of $K$, its $M$-symmetric tile belongs
to $K$ as well. In particular, $V_K$ is symmetric.
 \medskip

We shall see in Sect.~\SEC{symm-w-even} that such ftq-combi do exist, and
moreover, they just give rise to all maximal symmetric w-collections in
$2^{[n]}$. On the other hand, no ``symmetric ftq-combi'' can be devised when
$n$ is odd, as we explain in Sect.~\SEC{symm-w-odd}.

\subsection{Cubillages on cyclic zonotopes of dimention 3.} \label{ssec:cubillage}

Cubillages arising when we deal with chord separated collections live within a
3-dimensional $n$-colored cyclic zonotope. To define the latter, consider a set
$\Theta$ of $n$ vectors $\theta_i=(t_i,1,\phi(t_i))\in \Rset^3$,
$i=1,\ldots,n$, such that
  \begin{numitem1} \label{eq:Theta}
$t_1<t_2<\cdots <t_n$ and $\phi(t)$ is a strictly convex function; for example,
$\phi(t)=t^2$.
  \end{numitem1}
(Note that for our purposes, for a vector $v=(a,b,c)\in\Rset^3$, it is more
convenient to interpret $b$ as the vertical coordinate (\emph{height}), $a$ as
the left-to-right coordinate, and $c$ as the \emph{depth} of $v$. So all
vectors in $\Theta$ have the unit height.) \noindent An example with $n=5$ is
illustrated in the picture (where $z_i=x_i^2$ and $x_i=-x_{6-i}$).

  \vspace{0cm}
\begin{center}
\includegraphics{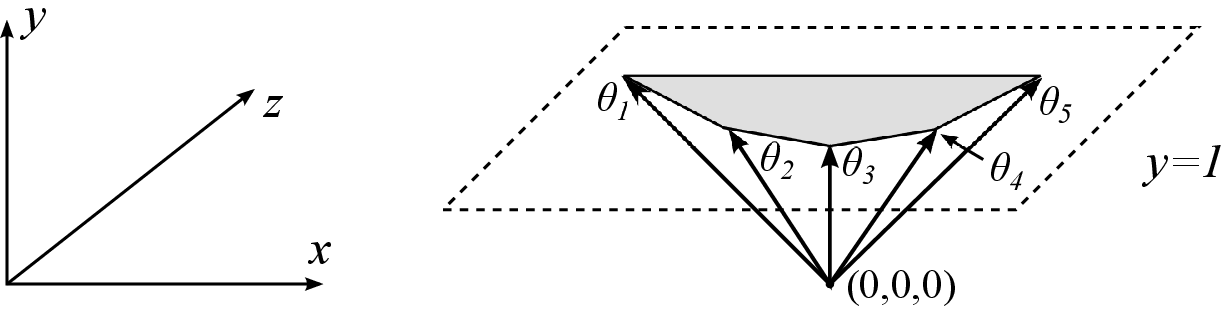}
\end{center}
\vspace{0cm}

The \emph{zonotope} $Z(\Theta)$ generated by $\Theta$ is the Minkowski sum of
line segments $[0,\theta_i]$, $i=1,\ldots,n$. Then a fine zonotopal tiling, or
a \emph{cubillage}, in terminology of~\cite{KV}, is (the polyhedral complex
determined by) a subdivision $Q$ of $Z(\Theta)$ into 3-dimensional
parallelotopes such that: any two intersecting ones share a common face, and
each face of the boundary of $Z(\Theta)$ is entirely contained in some of these
parallelotopes. For brevity, we refer to these parallelotopes as \emph{cubes},
and to $Q$ as a \emph{cubillage}.

Note that the choice of one or another cyclic configuration $\Xi$ (subject
to~\refeq{Theta}) is not important to us in essence, and we usually write
$Z(n,3)$ rather than $Z(\Theta)$, referring to it as the (cyclic 3-dimensional)
zonotope with $n$ colors.

Like the case of zonogons and combies, each vertex $v$ of a cubillage $Q$
(i.e., a vertex of some cube in it) is viewed as $\sum_{i\in X} \theta_i$ for
some subset $X\subseteq[n]$, and we identify such $v$ and $X$. The set of
vertices of $Q$ (as subsets of $[n]$) is called the \emph{spectrum} of $Q$ and
denoted as $V_Q$. One shows that $|V_Q|$ is equal to  $c_n$ as in~\refeq{gal},
and an important result due to Galashin establishes a relation of cubillages to
chord separation.

  \begin{theorem}[\cite{gal}] \label{tm:gal}
The correspondence $Q\mapsto V_Q$ gives a bijection between the set of
cubillages $Q$ on $Z(n,3)$ and the set ${\bf C}_n$ of maximal c-collections in
$2^{[n]}$.
  \end{theorem}

For a closed subset $U$ of points in $Z=Z(n,3)$, the \emph{front} (\emph{rear})
side of $U$, denoted as $\Ufr$ (resp. $\Urear$), is defined to be the set of
points $(a,b,c)\in U$ such that $c\le c'$ (resp. $c\ge c'$) among all
$(a',b',c')\in U$ with $(a',b')=(a,b)$, i.e., consisting of the points of $U$
with locally minimal (resp. maximal) depths. In particular, $\Zfr$ ($\Zrear$)
denotes the front (rear) side of the entire zonotope $Z$; it is well-known that
the vertices occurring in $\Zfr$ ($\Zrear$) are exactly the intervals (resp.
co-intervals) in $[n]$.
 \medskip

Next, to handle symmetric c-collections, we will deal with a \emph{symmetric}
set $\Theta$ of generating vectors $\theta_i=(t_i,1,\phi(t_i)))$, which means
that
  \begin{numitem1} \label{eq:sym-theta}
$t_i=-t_{i^\circ}$ and $\phi(t_i)=\phi(-t_i)$ for each $i\in[n]$,
  \end{numitem1}
and consider the corresponding \emph{symmetric zonotope} $Z(\Theta)$. It turns
out that, in contrast to the situation when symmetric (ftq)-combies exist only
for $n$ even, symmetric cubillages on a symmetric cyclic zonotope do exist in
both even and odd cases, as we shall see in Sects.~\SEC{symm-c-even}
and~\SEC{symm-c-odd}.


\section{Maximal symmetric w-collections: even case} \label{sec:symm-w-even}

In this section, we throughout assume that $n$ is even. Our goal is to prove
Theorem~\ref{tm:symm-ws} in this case. We consider a symmetric zonogon
$Z=Z(\Xi)\simeq Z(n,2)$, and an important role is played by the middle line $M$
in $Z$ (defined in~\refeq{M}).

We know (cf.~\refeq{sets-in-M}) that all points $A\subset[n]$ lying on $M$ are
self-symmetric, have size $n/2$, and admit only ordinary pairs. Consider two
distinct points $A,B$ in $M$. Since $|A|=|B|$, the symmetric difference
$A\triangle B$ ($=(A-B)\cup (B-A)$) has size at least 2. This is strengthened
as follows (this will be used in the next section):
 \begin{numitem1} \label{eq:neighAB}
if $A,B$ are self-symmetric and $|A\triangle B|=2$, then $A\triangle
B=\{i,i^\circ\}$ for some $i\in [n]$.
 \end{numitem1}
Indeed, if $A-B=\{i\}$, $B-A=\{j\}$ and $j\ne i^\circ$, then $j^\circ\in A-B$
(since each of $A,B$ contains exactly one elements of $\{j,j^\circ\}$). But
then $|A\triangle B|\ge 3$, a contradiction.
  \medskip

Next we prove the theorem (with $n$ even) as follows. Let $\Cscr$ be a
\emph{maximal by inclusion} symmetric w-collection in $2^{[n]}$ and suppose,
for a contradiction, that $|\Cscr|<s_n$. Extend $\Cscr$ to a maximal
(non-symmetric) w-collection $\Wscr\subset 2^{[n]}$. Then $|\Wscr|=s_n$, and in
view of~\refeq{ftq-combi}, there exists an ftq-combi $K$ on $Z$ whose vertex
set $V_K$ is exactly $\Wscr$.

Let $\Rscr=(R_0,R_1,\ldots,R_q)$  be the sequence of vertices of $K$ occurring
in $M$ and ordered from left to right. Note that $\Rscr$ is nonempty, since it
contains the vertex $[n/2]$ of the left boundary of $Z$, which is just $R_0$
(and the vertex $[(n/2+1)..n]$ of the right boundary of $Z$, which is $R_q$).
Consider two possible cases. (It should be noted that an idea of our analysis
in items~II and~III below is borrowed from Karpman's work~\cite{karp}.)
  \medskip

\noindent\underline{\emph{Case 1}:} Assume that the middle line $M$ of $Z$ is
covered by edges of $K$. Then (by the planarity and the construction of
ftq-combies) for each $p=1,\ldots,q$, the pair $e_p=(R_{p-1},R_p)$ forms an
edge of $K$. Moreover, $M$ separates the tiles of $K$ into two subsets $\Tscr$
and $\Tscr'$, where the former (latter) consists of the tiles lying in the half
of $Z$ below (resp. above) $M$. The intersection of these halves is just $M$,
and they are $M$-symmetric to each other. Now for each tile $\tau\in \Tscr$,
take the $M$-symmetric triangle $\tau^\ast$. Then the set $\Tscr''$ of such
tiles gives a subdivision of the half of $Z$ above $M$, and combining
$\Tscr$ and $\Tscr''$ (which have the same set of edges within $M$), we obtain a
symmetric ftq-combi $\tilde K$ on $Z$. Note also that if $A\in\Cscr$ is a
vertex in $\Tscr'$, then $A$ is a vertex of $\Tscr''$ as well (since the
symmetric set $A^\ast$ must be a vertex in $\Tscr$). Thus, $V_{\tilde K}$ is a
symmetric w-collection of size $s_n$ including $\Cscr$, contradicting the
maximality of $\Cscr$.
  \medskip

\noindent\underline{\emph{Case 2}:} Now assume that for some $1\le p\le q$, the
segment $\sigma$ of $M$ between the points $R_{p-1}$ and $R_p$ is not an edge
of $K$. Then there is a tile $\tau$ of $K$ with vertices $A,B,C$ such that one
vertex, $A$ say, is $R_{p-1}$, and $\sigma$ meets the edge connecting $B$ and $C$ at an interior point. Let for definiteness the point $B$ lies above $M$, and
$C$ below $M$. Our aim is to show that this is not the case.

A priori, $\tau$ can be one of the following shapes: $\Delta$-tile,
$\nabla$-tile, upper semi-lens, or lower semi-lens (defined in
Sect.~\SSEC{combi}). For reasons of symmetry, it suffices to consider the cases
when $\tau$ is either a $\nabla$-tile or an upper semi-lens. In the former case
$\tau$ is viewed as $\nabla(C|AB)$, while the latter case falls into two
subcases depending on the location of the vertex $A$; namely, $\tau$ is either
$U(ABC)$ or $U(CAB)$. So we have to consider three situations; they are illustrated in the picture (from left to right) and described in items I,\,II,\,III below. In order to analyze these situations, we use two auxiliary assertions.

\vspace{-0.2cm}
\begin{center}
\includegraphics[scale=1]{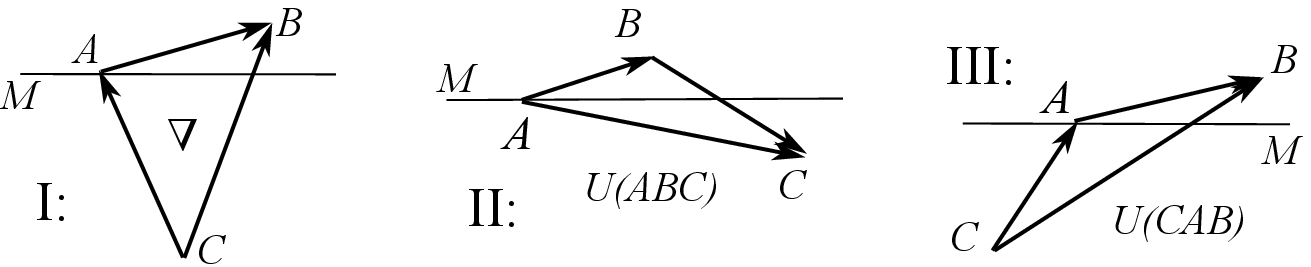}
\end{center}
\vspace{-0.2cm}

 \begin{numitem1} \label{eq:closeX}
If $S\subset[n]$ is such that each of $\Pi_0(S)$ and $\Pi_2(S)$ consists of at
most one pair, then $S$ and $S^\ast$ are weakly separated.
  \end{numitem1}
Indeed, we have $\Pi_2(S)=\{S-S^\ast\}$ and $\Pi_0(S)=\{S^\ast-S\}$
(cf.~\refeq{pi012}); so the assertion is immediate when some of $\Pi_0(S)$ and
$\Pi_2(S)$ is empty. And if $|\Pi_0(S)|=|\Pi_2(S)|=1$, then
$|S-S^\ast|=|S^\ast-S|=2$, whence $S$ and $S^\ast$ have the same size and one of them surrounds the other, again yielding the assertion.
  \begin{numitem1} \label{eq:XY}
If sets $S,T\subseteq[n]$ are weakly separated, then so are $S^\ast$ and
$T^\ast$.
  \end{numitem1}
Indeed, obviously, the weak separation of $S,T$ implies
that of $\bar S$ and $\bar T$, and that of $\{i^\circ \colon i\in S\}$ and
$\{i^\circ \colon i\in T\}$. Then the assertion follows from the fact that  the K-involution is the combination of the two involutions as in~\refeq{implicat}(i),(ii).

As a consequence of~\refeq{XY},
  \begin{numitem1} \label{eq:Ascr-S}
In a set $S\subseteq[n]$ is weakly separated from a symmetric w-collection
$\Cscr\subset 2^{[n]}$ and from the set $S^\ast$, then
$S^\ast$ is weakly separated from $\Cscr$ as well, implying that
$\Cscr\cup\{S,S^\ast\}$ is a symmetric w-collection.
  \end{numitem1}

I. We first consider the case $\tau=\nabla(C|AB)$. Then $|A|=|B|=m$ and $|C|=m-1$,
where $m=n/2$, and $A,B$ are expressed as $A=Ca$ and $B=Cb$ for some $a,b\in
[n]$. Obviously, $a<b$. Moreover, since $A$ lies on $M$, while $B$ above $M$,
we have $y_a=y_{a^\circ}<y_b=y_{b^\circ}$. Then, by the the concavity
condition~\refeq{additional}(i), the pair $\{a,a^\circ\}$ surrounds
$\{b,b^\circ\}$, and therefore
  \begin{numitem1} \label{eq:abba}
  ~either $a<b<b^\circ<a^\circ$ ~or ~$a<b^\circ<b<a^\circ$.
  \end{numitem1}

Note that the relations $A=Ca$, $B=Cb$ and $\Pi_0(A)=\Pi_2(A)=\emptyset$ give
  $$
  C\cap\{a,a^\circ,b,b^\circ\}=\{b^\circ\} \quad\mbox{and}\quad
  B\cap\{a,a^\circ,b,b^\circ\}=\{b,b^\circ\}.
  $$

Then $B^\ast\cap\{a,a^\circ,b,b^\circ\}=\{a,a^\circ\}$ (in view of
$\Pi_i(B^\ast)=\Pi_{2-i}(B)$). It follows that $C-B^\ast$ contains the element
$b^\circ$, whereas $B^\ast-C$ contains the elements $a,a^\circ$ surrounding
$b^\circ$ (by~\refeq{abba}). This together with $|B^\ast|=m>|C|$ implies that
$B^\ast$ and $C$ are not weakly separated. Therefore, $B^\ast$ is not in
$\Wscr$. On the other hand, since $\Pi_0(B)$ consists of the only pair
$\{a,a^\circ\}$, and $\Pi_2(B)$ of the only pair $\{b,b^\circ\}$, the sets
$B,B^\ast$ are weakly separated from each other (by~\refeq{closeX}). Also $B$ is  weakly separated from $\Cscr$ (since $\{B\}\cup \Cscr\subseteq V_K=\Wscr$). But then $\Cscr\cup\{B,B^\ast\}$ is a symmetric w-collection (by~\refeq{Ascr-S}), and now the maximality of $\Cscr$ implies $B,B^\ast\in \Cscr$, contrary to $B^\ast\notin\Wscr$.
 \medskip

II. Next we consider the case $\tau=U(ABC)$. Then $|A|=|B|=|C|$. Also $A=Xa$,
$B=Xb$ and $C=Xc$ for some $X\subset[n]$ and $a,b,c\in[n]$. Obviously, $a<b<c$.
Moreover, since $A$ lies on $M$, $B$ above $M$, and $C$ below $M$, it follows
from~\refeq{additional}(i) that $\{c,c^\circ\}$ surrounds $\{a,a^\circ\}$, and
the latter surrounds $\{b,b^\circ\}$. This is possible only if
 \begin{numitem1} \label{eq:cab}
 ~either $c^\circ<a<b<b^\circ<a^\circ<c$ ~or ~$c^\circ<a<b^\circ<b<a^\circ<c$.
 \end{numitem1}

Let $D:=\{a,a^\circ,b,b^\circ,c,c^\circ\}$. Then $\Pi_0(A)=\Pi_2(A)=\emptyset$ and  $a\in A\not\ni b,c$ imply $X\cap D=\{b^\circ,c^\circ\}$, whence
  $$
  B\cap D=\{b,b^\circ,c^\circ\}\quad \mbox{and}\quad C\cap
  D=\{b^\circ,c,c^\circ\}.
  $$

It follows that $B^\ast\cap D=\{a,a^\circ,c^\circ\}$. Then $C-B^\ast$ contains
$b^\circ,c$ and $B^\ast-C$ contains $a,a^\circ$, implying that $B^\ast$ and $C$
are not weakly separated since $a<b^\circ<a^\circ<c$ (cf.~\refeq{cab}). On the
other hand, one can see that $\Pi_0(B)=\{\{a,a^\circ\}\}$ and
$\Pi_2(B)=\{\{b,b^\circ\}\}$; therefore, $B$ and $B^\ast$ are weakly separated,
by~\refeq{closeX}. As in the previous case, we obtain that
$\Cscr\cup\{B,B^\ast\}$ is symmetric and weakly separated, yielding  $B,B^\ast\in\Cscr$ (by the maximality of $\Cscr$), contrary to $B^\ast\notin\Wscr$. 
\medskip

III. Finally, consider the case $\tau=U(CAB)$. Then $A=Xa$, $B=Xb$ and $C=Xc$
for some $X\subset[n]$ and elements $c<a<b$. Arguing as above, we observe that
 \begin{numitem1} \label{eq:cab2}
 ~either $c<a<b<b^\circ<a^\circ<c^\circ$ ~or ~$c<a<b^\circ<b<a^\circ<c^\circ$,
 \end{numitem1}
and for $D:=\{a,a^\circ,b,b^\circ,c,c^\circ\}$, we have
  $$
  B\cap D=\{b,b^\circ,c^\circ\},\quad B^\ast\cap D=\{a,a^\circ,c^\circ\}\quad
  \mbox{and}\quad C\cap D=\{b^\circ,c,c^\circ\}.
  $$

Then $C-B^\ast$ contains $b^\circ,c$ and $B^\ast-C$ contains $a,a^\circ$,
whence $B^\ast$ and $C$ are not weakly separated, in view of
$c<a<b^\circ<a^\circ$ (cf.~\refeq{cab2}). On the other hand, $\Cscr\cup\{B,B^\ast\}$ is symmetric and weakly separated, yielding  $B,B^\ast\in\Cscr$. 

This completes the proof of Theorem~\ref{tm:symm-ws} when $n$ is even. \hfill
\qed\qed
\medskip

\noindent\textbf{Remark 3.} Assertion~\refeq{symm-str} on the purity of
symmetric strongly separated collections in $2^{[n]}$ with $n$ even is proved
in a similar way (and even simpler, using an observation that if $A\subseteq[n]$ is strongly separated from $A^\ast$, then at least one of $\Pi_0(A)$ and $\Pi_2(A)$ must be empty). On this way, given a maximal by inclusion symmetric
s-collection $\Ascr\subset 2^{[n]}$, we extend it to a maximal s-collection
$\Sscr$ and take the combi $K$ without lenses (equivalent to a rhombus tiling)
with $V_K=\Sscr$. The situation when the middle line $M$ is not fully covered
by edges of $K$ is again impossible (now an analysis of the only case
$\tau=\nabla(C|AB)$ is sufficient, repeating part~I of the above proof). And
when $M$ is covered by edges of $K$, we replace the subcombi of $K$ above $M$
in a due way (as described in Case~1 of the proof), obtaining a symmetric combi without lenses (viz. rhombus tiling) whose vertex set includes
$\Ascr$. This gives~\refeq{symm-str} in the even case:
 \begin{numitem1} \label{eq:symm-even-ss}
when $n$ is even, all inclusion-wise maximal symmetric s-collections in
$2^{[n]}$ have the same cardinality, which is equal to $s_n$.
 \end{numitem1} 

We finish this section with one more assertion that will be used in the next
section:
    \begin{numitem1} \label{eq:n/2}
for $n$ even, if an ftq-combi $K$ on the symmetric zonogon $Z(n,2)$ has a path
$P$ covering the middle line $M$, then this path contains exactly $n/2$ edges.
  \end{numitem1}

Indeed, let $R_0,R_1,\ldots,R_q$ be the sequence of vertices of $P$, and let $e_p$ denote
the edge from $R_{p-1}$ to $R_p$. Then $|R_{p-1}\triangle R_p|=2$, and
by~\refeq{neighAB}, $e_p$ is congruent to the vector
$\eps_{ii^\circ}=\xi_{i^\circ}-\xi_{i}$ for some $i\in[n/2]$. The sum of these
vectors over $M$ is equal to the difference of $R_q$ and $R_0$ (regarded as
vectors), namely, $\sum_{i=n/2+1}^n \xi_i-\sum_{i=1}^{n/2}\xi_i$. This is just
equal to $\sum(\eps_{ii^\circ}\colon i\in[n/2])$, whence the result easily
follows.


\section{Maximal symmetric w-collections: odd case} \label{sec:symm-w-odd}

In this section we prove Theorem~\ref{tm:symm-ws} when $n$ is odd, $n=2m+1$. It
is convenient to deal with the set of colors in the symmetrized form, using
notation $[-m..m]$ ($=\{-m,\ldots,-1,0,1,\ldots,m\}$) introduced in
Sect.~\SEC{prelim}. The proof is based on a reduction to the above result with
an even number of colors, namely, $[-m..m]^{-}$
($=\{-m,\ldots,-1,1,\ldots,m\}$).

Let $\Cscr$ be a symmetric w-collection in $2^{[-m..m]}$. Consider the
partition $\Cscr=\Cscr'\sqcup \Cscr''$, where
  $$
  \Cscr':=\{A\in \Cscr\colon 0\notin A\}\quad \mbox{and} \quad
  \Cscr'':=\{A\in \Cscr\colon 0\in A\}.
  $$
Note that for any $A\subseteq [-m..m]$, ~$|A|+|A^\ast|=2m+1$ and exactly one of
$A,A^\ast$ contains the element $0$ (in view of $0^\circ=0$). So for
$A\in\Cscr'$, its symmetric set $A^\ast$ belongs to $\Cscr''$, and vice versa.
Also $A$ and $A^\ast$ coincide on the ordinary pairs $\{i,i^\circ\}$ (i.e.,
such that $i\ne 0$ and $|\{i,i^\circ\}\cap A|=1$), and are complementary on the poor and full pairs:
  $$
  \Pi_0(A^\ast)=\Pi_2(A) \quad\mbox{and} \quad \Pi_2(A^\ast)=\Pi_0(A);
  $$
cf.~\refeq{pi012} (recall that the ``pair'' $\{0,0^\circ\}$ is regarded as poor or full). Observe that
 \begin{numitem1} \label{eq:AAm}
$|A|\le m$ if $A\in\Cscr'$, and $|A|\ge m+1$ if $A\in\Cscr''$.
  \end{numitem1}
Indeed, let $A\in\Cscr'$; then $0\in A^\ast-A$. If $A$ has no full pair, then $|A|<|A^\ast|$ follows from the fact that the ordinary pairs for $A$ and $A^\ast$ are the same. And if $A$ contains a full pair, then this pair  belongs to $A-A^\ast$ and surrounds the element 0. Since $A,A^\ast$ are weakly separated, we have $|A|\le|A^\ast|$ (and this inequality is strict since $|A|+|A^\ast|=2m+1$).

In particular, $|A|=m$ and $|A^\ast|=m+1$ if $\Pi_2(A)=\emptyset$ and $\Pi_0(A)=\{\{0,0^\circ\}\}$; in this case, we say that the pair $\{A,A^\ast\}$ is \emph{squeezed} (an analog of self-symmetric sets for $n$ even).

Now form the collections $\Dscr',\Dscr'',\Dscr$ of subsets of $[-m..m]^-$ as
  $$
  \Dscr':=\Cscr',\quad \Dscr'':=\{A-0\colon A\in\Cscr''\},\quad
  \mbox{and} \quad
  \Dscr:=\Dscr'\cup\Dscr''.
  $$

For convenience, we will denote the K-involution on sets in $[-m..m]^-$ with
symbol $\natural$ (to differ from the K-involution $\ast$ for $[-m..m]$). We
observe that
  \begin{numitem1} \label{eq:Dsymm-ws}
the collection $\Dscr$ is $\natural$-symmetric (i.e., symmetric w.r.t. $\natural$) and weakly separated.
  \end{numitem1}
Indeed, $\Cscr$ is partitioned into symmetric pairs $\{A,A^\ast\}$, where
$A\in\Cscr'$ and $A^\ast\in\Cscr''$. Then $A\in \Dscr'$ and
$A^\ast-0\in\Dscr''$. One can see that $A^\ast-0$ is just $A^\natural$.
Therefore, $\Dscr$ is $\natural$-symmetric. Next, let $A,B\in\Dscr$. Obviously, $A,B$ are
weakly separated if both are either in $\Dscr'$ or in $\Dscr''$. So assume that
$A\in\Dscr'$ and $B\in\Dscr''$. Then $A\in\Cscr'$ and $B0\in\Cscr''$. Since
$\Cscr$ is a w-collection and $|A|\le m<|B0|$ (by~\refeq{AAm}), either $A$ and
$B0$ are strongly separated, or they are weakly separated and $A$ surrounds
$B0$. In the former case, $A$ and $B$ are strongly separated, while in the
latter case, the inequality $|A|\le|B|$ ensures that $A$ and $B$ are weakly
separated.

We call $\Dscr$ the \emph{contraction} of $\Cscr$, and call the operation of
getting rid of color 0 as above the \emph{contraction operation} on $\Cscr$. A
converse operation handles a symmetric w-collection $\Dscr$ in $2^{[-m..m]^-}$
and lifts it to a collection $\Escr$ in $2^{[-m..m]}$, as follows.

Represent the sets $A\in\Dscr$ as points in the symmetric zonogon $Z=Z(\Xi)$
with $\Xi=\{\xi_{-m},\ldots,\xi_{-1},\xi_1,\ldots,\xi_m\}$ and define $\Dscr'$
($\Dscr''$) to consist of those points $A\in\Dscr$ that lie below the middle
line $M$ or on $M$ (resp. above $M$ or on $M$). They determine the collections
$\Escr',\Escr''$ in $2^{[-m..m]}$ by
  $$
  \Escr':=\Dscr'\quad \mbox{and} \quad \Escr'':=\{A0\colon A\in\Dscr''\}.
  $$
In particular, if $A\in\Dscr$ is self-symmetric (viz. lies on $M$), then
$A\in\Dscr'\cap\Dscr''$ and $A$ determines the squeezed pair $\{A,A^\ast=A0\}$
in $[-m..m]$.

We call $\Escr:=\Escr'\cup\Escr''$ the \emph{expansion} of $\Dscr$ using color
0. The following properties are valid:
  \begin{numitem1} \label{eq:Esymm-ws}
$\Escr$ is symmetric and weakly separated;
  \end{numitem1}
 \begin{numitem1} \label{eq:CDE}
if $\Cscr\subset 2^{[-m..m]}$ is a symmetric w-collection, $\Dscr$ is the
contraction of $\Cscr$, and $\Escr$ is the expansion of $\Dscr$ using color 0,
then $\Escr=\Cscr$.
  \end{numitem1}

We check these properties and simultaneously finish the proof of the theorem by
using the geometric construction from the previous section.

More precisely, we extend the contraction $\Dscr$ of $\Cscr$ to a maximal
$\natural$-symmetric w-collection $\Wscr$ in $2^{[-m..m]^-}$ and take a $\natural$-symmetric
ftq-combi $K$ on the zonogon $Z=Z(\Xi)$ such that $V_K=\Wscr$. Let
$\Dscr',\Dscr''$ be defined as above. Then $\Dscr'$ represents a subset of
vertices of $K$ in the lower half $\Zlow$ of $Z$ (up to $M$), and $\Dscr''$ is
the set $M$-symmetric to $\Dscr'$, which lies in the upper half $\Zup$ of $Z$.

Let $\Klow$ and $\Kup$ be the parts (subcomplexes) of $K$ contained in $\Zlow$
and $\Zup$, respectively. Then $\Klow\cap\Kup$ gives a directed path $P$ on $M$
consisting of $m+1$ vertices and $m$ edges, say, $P=R_0R_1\cdots R_m$
(cf.~\refeq{n/2}). Moreover, by~\refeq{neighAB}, each edge $e_p:=(R_{p-1},R_p)$
is of type $ii^\circ$ for some $-m\le i\le-1$.

Now consider the larger zonogon $Z'=Z(\Xi')$, where $\Xi'$ is obtained by
adding to $\Xi$ the vertical vector $\xi_0=(0,y_0)$. Equivalently, $Z'$ is
formed by splitting $Z$ along $M$, keeping the part $\Zlow$ of $Z$, moving the
part $\Zup$ by $y_0$ units in the vertical direction, and filling the gap
between $\Zlow$ and $\Zup+\xi_0$ by the rectangle $F$ congruent to
$M\times\xi_0$. Accordingly, the part $\Kup$ of $K$ is moved by $\xi_0$, thus
transforming each vertex $A$ of $\Kup$ into $A0$, and we subdivide $F$ into the
sequence of rectangles $F_1,\ldots,F_m$, where $F_p$ is congruent to $e_p\times
\xi_0$. This results in a ``pseudo-combi'' $K'$, with the natural involution on
the vertices which brings each $A\in V_{K'}$ to its $M'$-symmetric vertex $A'$,
where $M'$ is the updated middle line $M'$ with $y^{M'}=y^M+y_0/2$. Clearly
$A'=A^\ast$, and the converse transformation $Z'\mapsto Z$ returns the
ftq-combi $K$.

Finally, we can transform $K'$ into a correct, though not symmetric, ftq-combi
$K''$ on $Z'$, by subdividing each rectangle $F_p$ into four triangles. Namely,
using the fact that the edge $e_p$ of $P$ has type $ii^\circ$ for some $-m\le
i\le -1$, the devised triangles are viewed as (using notation from
Sect.~\SSEC{combi})
   \begin{equation} \label{eq:rect-4triang}
   \nabla(R_{p-1}|R^\ast_{p-1} v_p),\quad \nabla(R_{p}|v_p R^\ast_p), \quad
   \Delta(v_p|R_{p-1}R_p), \quad L(R^\ast_{p-1} v_p R^\ast_p),
   \end{equation}
where $v_p$ is the corresponding point (viz. subset of $[-m..m]$) in the
interior of $F_p$ (so that the edge $(R_{p-1},v_p)$ has color $i^\circ$, and
$(R_p,v_p)$ color $i$). See the picture.

\vspace{-0.2cm}
\begin{center}
\includegraphics[scale=1.0]{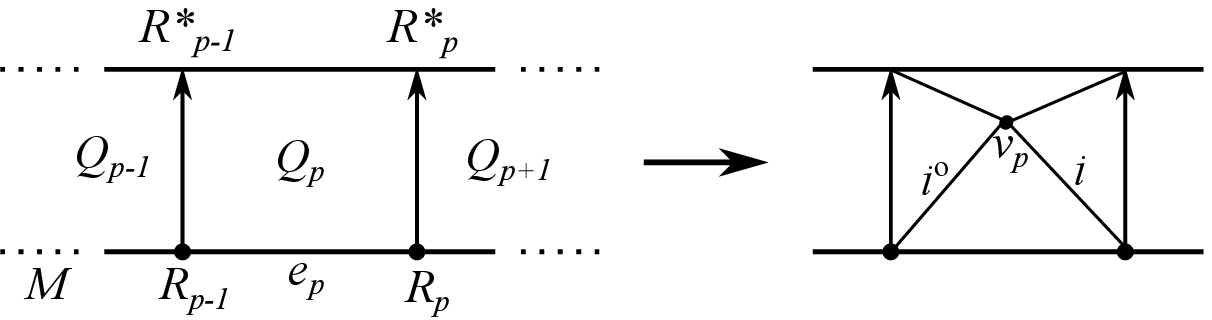}
\end{center}
\vspace{-0.2cm}

As a result, the initial collection $\Cscr$ is included in
$\Wscr':=V_{K''}-\{v_1,\ldots, v_m\}$ and coincides with the expansion $\Escr$
of $\Dscr$, yielding~\refeq{Esymm-ws} and~\refeq{CDE}. The collection $\Wscr'$
in $2^{[-m..m]}$ is symmetric and weakly separated. It has size
$|V_{K''}|-m=s_{2m+1}-m$, and the result follows. \hfill\qed\qed
 \medskip

\noindent\textbf{Remark 4.} The above method can be applied (in a simpler form) to \emph{strongly} separated collections for $n$ odd. Namely, for a symmetric s-collection $\Cscr\subset 2^{[-m..m]}$, we form the collections $\Cscr',\Cscr'',\Dscr',\Dscr'',\Escr',\Escr''$ as described above. (Note that properties~\refeq{AAm} and~\refeq{Dsymm-ws} easily follow from the fact that $\Pi_2(A)=\emptyset$ for each $A\in\Cscr'$, which is provided by the strong separation of $\Cscr$.) Extending the contraction $\Dscr=\Dscr'\cup\Dscr''$ of $\Cscr$ to a maximal symmetric s-collection $\Sscr$ in $2^{[-m..m]^-}$, we take the corresponding symmetric rhombus tiling $T$ on $Z$ with $V_T=\Sscr$. Acting as above, we move the ``upper'' parts of $Z$ and $T$ (lying above $M$) by the vector $\xi_0$ and for each vertex $R_i$ on $M$, add the vertical  edge $u_i$ connecting $R_i$ and $R_i+\xi_0$. The difference with the above construction for ftq-combies is that, instead of replicating the horizontal edges $(R_{i-1},R_i)$ lying on $M$, we now simply split each symmetric rhombus between $R_{i-1}$ and $R_i$ into two triangles and move the upper one (of $\Delta$ type) by $\xi_0$. This together with the vertical edges $u_{i-1}$ and $u_i$ produces a symmetric hexagon. As a result, we obtain a ``pseudo'' rhombus tiling $T'$ in which the middle tiles are formed by symmetric hexagons, not rhombi. The transformation $T\mapsto T'$ is illustrated in the picture where $m=2$. Note that $T'$ can be transformed into a rhombus tiling (which is not symmetric) by subdividing each middle hexagon into three rhombi (by one of two possible ways); such a subdivision is shown  by dotted lines in the picture. Note that the subdivision adds $m$ new vertices.
 
 \vspace{-0.2cm}
\begin{center}
\includegraphics[scale=0.8]{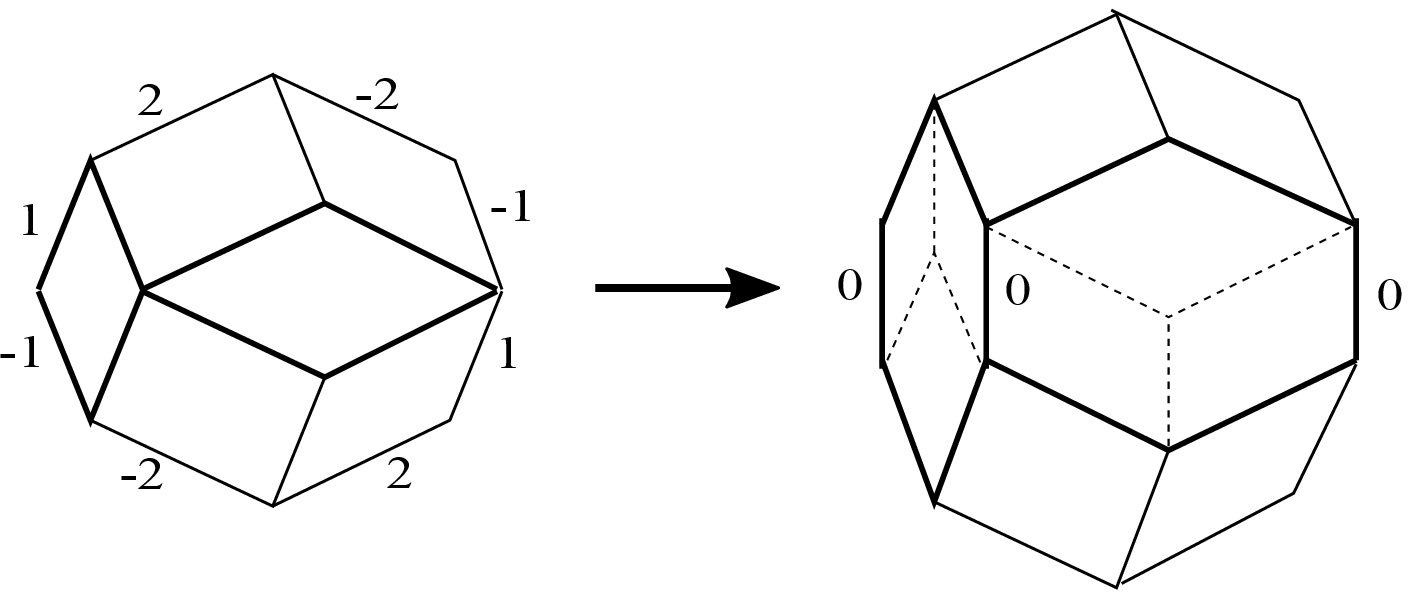}
\end{center}
\vspace{-0.2cm}

As a result, we obtain assertion~\refeq{symm-str} in the odd case, namely:
 \begin{numitem1} \label{eq:symm-odd-ss}
when $n$ is odd, all inclusion-wise maximal symmetric s-collections in
$2^{[n]}$ have the same cardinality, equal to $s_n-(n-1)/2$.
 \end{numitem1}


\section{Maximal symmetric c-collections: even case} \label{sec:symm-c-even}

In this section we prove Theorem~\ref{tm:symm-chord} when the number $n$ of
colors is even, $n=2m$. Our method of proof uses a reduction to symmetric
weakly separated collections and combies. We will deal with the set of colors
given in the symmetrized form, namely, $\mmmm$ (which means
$\{-m,\ldots-1,1,\ldots,m\}$).

Let $\Cscr$ be a maximal symmetric chord separated collection in $2^\mmmm$. We
have to show that $|\Cscr|=c_{2m}$ (where $c$ is defined in~\refeq{gal}).
Suppose, for a contradiction, that this is not so: $|\Cscr|<c_{2m}$. Extend
$\Cscr$ to a (non-symmetric) c-collection $\Fscr\subseteq 2^\mmmm$ of size
$c_{2m}$ and represent $\Fscr$ as the spectrum (vertex set) $V_{Q}$ of a
cubillage $Q$ in the corresponding zonotope; such $\Fscr$ and $Q$ exist by
Galashin's result (Theorem~\ref{tm:gal}).

More precisely, we consider the zonotope $Z=Z(\Theta)$ generated by the set
$\Theta$ of vectors $\theta_i=(t_i,1,\phi(t_i))$, $i\in\mmmm$, subject
to~\refeq{Theta}. Recall that the second coordinate of a point in $\Rset^3$ is
thought of as the height of this point. From the symmetry conditions on
$\Theta$ (namely, $i^\circ=-i$, $t_{i^\circ}=-t_i$ and
$\phi(t_{i^\circ})=\phi(t_i)$) it follows that:
  \smallskip

(a) $Z$ is centrally symmetric w.r.t. the point
$\zeta_Z$ that is the half sum of vectors in $\Theta$, called the
\emph{center} of $Z$ (i.e., $\zeta_Z=(0,m,\phi(t_1)+\cdots+\phi(t_m))$);
more precisely, $v=(a,b,c)\in Z$ implies $\omega(v):= 2\zeta_Z-(a,b,c)\in Z$;
and
\smallskip

(b) $Z$ is symmetric w.r.t. the vertical plane $\{(a,b,c)\in\Rset^3\colon
a=0\}$; namely, $v=(a,b,c)\in Z$ implies $\nu(v):=(-a,b,c)\in Z$.
  \medskip

Let $L$ be the line segment in $Z$ going through the center $\zeta_Z$
orthogonal to the plane as in~(b). Then $L$ connects the vertices of $Z$
representing the sets (intervals) $[-m..-1]$ and $[1..m]$; we call $L$ the
\emph{axis} of $Z$. The involutions $\omega$ and $\nu$ as in~(a) and~(b) commute, and taking their composition, we have one more involution $\mu$ on $Z$; it preserves the first coordinate and gives a symmetry w.r.t. $L$. A nice property of $\mu$ (which is easy to verify) is that
  \begin{numitem1} \label{eq:mu}
$\mu$ sends each set $A\subseteq\mmmm$ (regarded as a point) to the symmetric
set $A^\ast$; in particular, $A$ lies on the axis $L$ if and only if $A$ is
self-symmetric: $A=A^\ast$.
  \end{numitem1}

So $\mu$ can be thought of as a linear extension of the K-involution on
$2^\mmmm$; we call it the \emph{geometric K-involution}. (Note also that $\mu$
swaps the front side $\Zfr$ and the rear side $\Zrear$ of $Z$.)

One more object important to us is the section $\Omega$ of $Z$ by the
horizontal plane $H_m:=\{(a,b,c)\in \Rset^3\colon b=m\}$; it contains the
center $\zeta_Z$ and axis $L$ and we call it the \emph{plate} in $Z$. This $\Omega$
divides $Z$ into two halves $\Zlow$ and $\Zup$, which lie below and above
$\Omega$, respectively, and intersect by $\Omega$. The symmetry $\mu$
swaps $\Zlow$ and $\Zup$.

In its turn, the axis $L$ divides the plate (disk) $\Omega$ into two halves, symmetric
to each other by $\mu$, and the boundary of $\Omega$ is partitioned into two
piece-wise linear paths $\Omegafr$ and $\Omegarear$ connecting the vertices
$[-m..-1]$ and $[1..m]$, where the former lies in $\Zfr$, and the latter in
$\Zrear$.

In fact, we wish to transform the cubillage $Q$ as above into a symmetric
cubillage $Q'$ keeping the collection $\Cscr$ in its spectrum, where we call a
cubillage on $Z$ symmetric if it is stable under $\mu$. Then $V_{Q'}$ is
symmetric, whence $V_{Q'}=\Cscr$, and we are done.

The task of constructing the desired $Q'$ is reduced to handling certain weakly
separated collections and ftq-combies. This relies on a method involving weak
membranes in fragmentized cubillages developed in~\cite[Sect.~6]{DKK3}. We now interrupt our description for a while to briefly review the notions and constructions
needed to us.
\medskip

\noindent \textbf{Fragmentation and weak membranes.} For an arbitrary $n$, let $Q$ be a cubillage on the zonotope $Z\simeq Z(n,3)$
generated by vectors $\theta_1,\ldots,\theta_n$ as in~\refeq{Theta}. A cube $C$ in $Q$ may be denoted as $(X|T)$, where $X\subset[n]$ is the lowest vertex, and $T$ the triple of edge colors, $i<j<k$ say, in $C$.

By the \emph{fragmentation} of $Q$ we mean the complex $\Qfrag$ obtained by
cutting $Q$ by the horizontal planes $H_\ell:=\{(a,b,c)\in\Rset^3\colon
b=\ell\}$ for $\ell=1,\ldots,n-1$. These planes subdivide each cube $C=(X|T)$
into 3 pieces $\Cfrag_1,\Cfrag_2,\Cfrag_3$, where $\Cfrag_h=C_h$ is the portion of
$C$ between $H_{|X|+h-1}$ and $H_{|X|+h}$. So $\Cfrag_1$ is a simplex with
the bottom vertex $X$, $\Cfrag_3$ is a simplex with the top vertex $X\cup T$,
and $\Cfrag_2$ is an octahedron; see the picture (where the objects are slightly slanted). 

\vspace{-0.2cm}
\begin{center}
\includegraphics[scale=1.0]{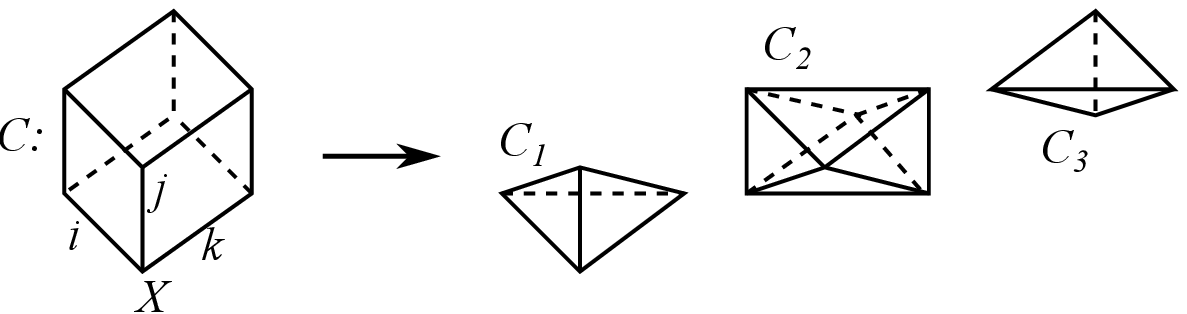}
\end{center}
\vspace{-0.2cm}

A \emph{facet} of $\Qfrag$ is meant to be a
facet of any \emph{fragment} $\Cfrag_h$ of a cube $C$ in $Q$. This is a
triangle of one of two sorts: either a horizontal triangle (section)
$S_h(C):=C\cap H_{|X|+h}$, $h=1,2$, or a half of the parallelogram forming a face
of $C$, that we conditionally call a ``vertical'' triangle. Note that the triple of vertices of a horizontal facet $F$ are of the form either $Xi,Xj,Xk$ or $Y-i,Y-j,Y-k$ for some $X,Y\subseteq[n]$ and $i<j<k$, and $i,j,k$ are just the colors of the cube containing $F$ as a section. This implies that
 \begin{numitem1} \label{eq:hor-sect}
any horizontal facet $F$ in $\Qfrag$ determines both fragments sharing $F$, as well as the cube separated by $F$ (among all cubillages whose fragmentation contains $F$).
 \end{numitem1}
(This need not be so for a vertical facet.)
Consider the projection $\pi^\rho:\Rset^3\to\Rset^2$ defined by
  \begin{numitem1} \label{eq:project}
$(a,b,c)\mapsto (x,y)$, where $x:=a$ and $y:=b-\rho c$
  \end{numitem1}
for a sufficiently small real $\rho>0$. Observe that $\pi^\rho$ maps the
generators $\theta_1,\ldots,\theta_n$ of the zonotope $Z=Z(\Theta)\simeq
Z(n,3)$ to generators $\xi_1,\ldots,\xi_n$ (respectively) as in
Sect.~\SSEC{combi}; we may assume, w.l.o.g., that the latter generators
satisfy~\refeq{xi-x-y},\refeq{additional}, and therefore, the image by
$\pi^\rho$ of $Z$ is the zonogon $Z'=Z(\Xi)\simeq Z(n,2)$ (here the concavity condition~\refeq{additional}(i) is provided by the convexity of $\phi$ in~\refeq{Theta} and the relation $\rho>0$). Note that for points
$(a,b,c),(a,b,c')\in\Rset^3$ if $c<c'$, then their images $(x,y)$ and $(x',y')$
(respectively) satisfy $x=x'$ and $y>y'$. This implies
that under the projection $\pi^\rho$, \emph{each horizontal facet of $\Qfrag$ becomes ``fully seen'' (slightly from behind), and therefore forms a non-degenerate triangle in $Z'$}.
  \medskip

\noindent\textbf{Definition.} A 2-dimensional subcomplex (``surface'') $N$ in
$\Qfrag$ is called a \emph{weak membrane}, or a \emph{w-membrane} for short, of
$Q$ if $\pi^\rho$ bijectively projects $N$ onto $Z'$. When $N$ has only
vertical triangles (but no horizontal ones), $N$ is called a \emph{strong
membrane}.
  \medskip

Additional explanations: Let $N$ contain a facet $F$ of a fragment $\Cfrag_h$
of a cube $C=(X\,|\,T=(i<j<k))$. Then: (a) if $F$ is the section $S_1(C)$
(resp. $S_2(C)$), then $\pi^\rho(F)$ is an upper (resp. lower) semi-lens with edges
of types $ij,jk,ik$; and (b) if $F$ is the lower (upper) half of a facet
(``rhombus'') of $C$, then $\pi^\rho(F)$ is a $\nabla$-tile (resp.
$\Delta$-tile) with the same edge colors as those of $F$. 

As is explained in~\cite[Sects.~6.3,7]{DKK3}, w-membranes are closely related to ftq-combies:
  \begin{numitem1} \label{eq:membr-combi}
for a cubillage $Q$ on $Z$ and a w-membrane $N$ in $\Qfrag$, the projection $\pi^\rho(N)$ (regarded as a complex) forms an ftq-combi on the zonogon $Z'$, and conversely,
for any ftq-combi $K$ on $Z'$, there exist a cubillage $Q$ on $Z$ and a
w-membrane $N$ in $\Qfrag$ such that $\pi^\rho(N)=K$.
  \end{numitem1}

Particular cases of membranes are the front side $\Zfr$ and the rear side
$\Zrear$ of $Z$ (see Sect.~\SSEC{combi}), regarded as complexes in which each
facet (``rhombus'') is subdivided (fragmented) into two halves; both membranes are strong.

The picture below illustrates the simplest case when the cubillage consists of only one cube $C$, i.e., $n=3$ and $Z=C$; here there are four w-membranes $N_1\simeq\Cfr$, $N_2,\,N_3,\,N_4\simeq\Crear$, and the horizontal facets in $N_2$ and $N_3$ are dark.

\vspace{-0.2cm}
\begin{center}
\includegraphics[scale=0.9]{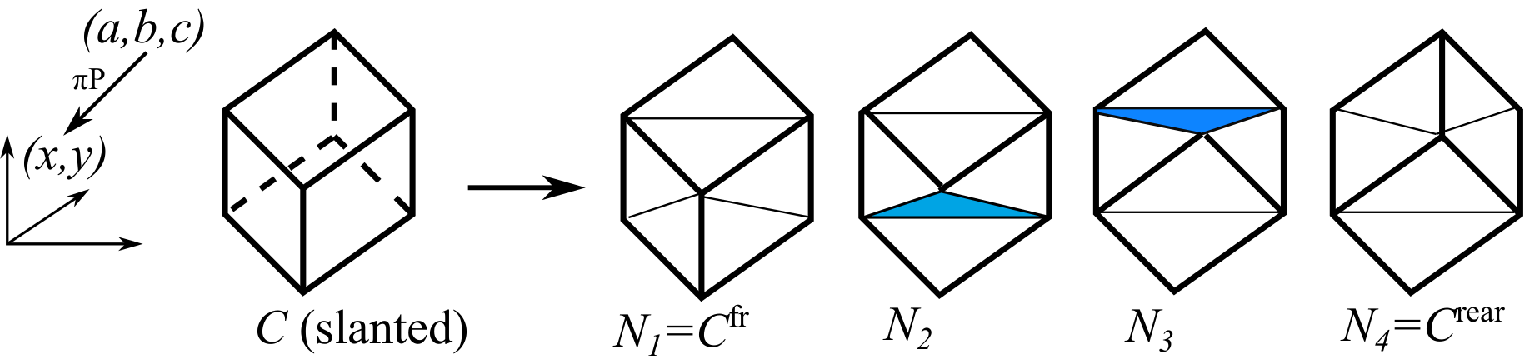}
\end{center}
\vspace{-0.2cm}

Next we return to the initial cubillage $Q$ with $\Cscr\subset V_Q$ on the
symmetric zonotope $Z=Z(\Theta)$ with $n=2m$ colors. In its fragmentation $\Qfrag$, we distinguish
the particular weak membrane $\Ndiam$ as follows:
  \begin{numitem1} \label{eq:Ndiam}
$\Ndiam=\Zfr_\uparrow \cup Q_\Omega\cup \Zrear_\downarrow$, where $Q_\Omega$ is formed by the facets of $\Qfrag$ lying on the plate $\Omega$;
~$\Zfr_\uparrow$ is the part of the front side $\Zfr$  of $Z$ contained in the upper half $\Zup$; and $\Zrear_\downarrow$ is the part of the rear side $\Zrear$ contained in the lower half $\Zlow$.
  \end{numitem1}

This $\Ndiam$ is indeed a w-membrane (taking into account that the plate $\Omega$ becomes ``fully seen'' under $\pi^\rho$, whence the projection $\pi^\rho$ is injective on $\Ndiam$; the fact that $\pi^\rho(\Ndiam)=Z'$ is easy). Also one can see that $\Zfr_\uparrow$ is symmetric to $\Zrear_\downarrow$. The pieces (disks) $\Zfr_\uparrow$ and $Q_\Omega$ share
the path $\Omegafr$ (connecting the vertices $[-m..-1]$ and $[1..m]$ and lying on $\Zfr$), while $Q_\Omega$ and $\Zrear_\downarrow$ share
the rear boundary path $\Omegarear$ of $\Omega$. Then $K:=\pi^\rho(\Ndiam)$ is
an ftq-combi on $Z'\simeq Z(2m,2)$ in which the part $\pi^\rho(\Zfr_\uparrow)$
is symmetric to $\pi^\rho(\Zrear_\downarrow)$, and the image by $\pi^\rho$ of
the axis $L$ of $Z$ is just the middle line $M$ of $Z'$.

In the rest of the proof we rely on the simple fact analogous to~\refeq{XY} for
the weak separation: if $A,B\subseteq \mmmm$ are chord separated, then so are
$A^\ast,B^\ast$. This implies:
  \begin{numitem1} \label{eq:C-AA}
for the maximal symmetric c-collection $\Cscr$ as above, if $A\in 2^\mmmm$ is
chord separated from $\Cscr$ and from $A^\ast$, then $A,A^\ast\in\Cscr$.
  \end{numitem1}
Indeed, one can see that $A^\ast$ is chord separated from $\Cscr$ (using the
symmetry of $\Cscr$). Then $\Cscr\cup\{A,A^\ast\}$ is chord separated, and the maximality of $\Cscr$ implies $A,A^\ast\in\Cscr$.

Next we proceed as follows. By reasonings in Sect.~\SEC{symm-w-even}, the middle line $M$ in $K=\pi^\rho(\Ndiam)$ is covered by edges of $K$, and
combining the part $\Klow$ of $K$ below $M$ with its $M$-symmetric complex
$(\Klow)^\ast$ (where $\ast$ stands for the symmetry in $Z'$), we obtain a
correct symmetric ftq-combi $K'$ on $Z'$. (Note that the part
$\pi^\rho(\Zfr_\uparrow)$ of $\Klow$ is already symmetric to the part
$\pi^\rho(\Zrear_\downarrow)$ of $\Kup$; so $K'$ and $K$ coincide within these parts.) Each set $A\in V_{\Klow}$ belongs to the spectrum $V_Q$, and
therefore it is chord separated from $\Cscr$. Also the sets $A,A^\ast$ are
chord separated (since both belong to the same ftq-combi $K'$, whence they are even weakly separated). Then $A,A^\ast\in \Cscr$, by~\refeq{C-AA}.
It follows that
 \begin{numitem1} \label{eq:QOmega} 
the set of \emph{vertices} of $Q$ lying on the
plate $\Omega$, i.e., $V_{Q_\Omega}$, is self-symmetric (w.r.t. $L$) and
belongs to $\Cscr$, implying that the vertex set of $K$ is symmetric. 
  \end{numitem1}
  
\noindent\textbf{Remark 5.} Note that a priori~\refeq{QOmega} does not guarantee that the subcomplex $Q_\Omega$ itself (or the set of triangles it it) 
is self-symmetric. (The reason is that in a quasi-combi, a semi-lens with more than three vertices can be triangulated in several different ways.) If this were so, then the proof could be finished as follows.
Take the part $\Qfrag_\downarrow$ of $\Qfrag$ below $\Omega$ (which subdivides
$\Zlow$) and replace the part of $\Qfrag$ above $\Omega$ by the complex
symmetric to $\Qfrag_\downarrow$. These two complexes coincide
within the plate $\Omega$ (due to the symmetry of $Q_\Omega$), and we can conclude that their union gives the correct fragmentation of some symmetric
cubillage $Q'$ on $Z$. (Here we rely on the fact, due to~\refeq{hor-sect}, that  any horizontal triangle in $Q_\Omega$ uniquely determines the cube having this triangle as a section.) Then $V_{Q'}\supseteq\Cscr$ would imply the result.
  \medskip

We argue in a different way, using a trick of \emph{perturbing}
the plate $\Omega$ described below.

To slightly simplify our description, we identify the w-membrane
$\Ndiam$ and its image by $\pi^\rho$, the ftq-combi $K=\pi^\rho(\Ndiam)$. Let
$\Gamma=(\Vscr,\Escr)$ be the graph whose vertices are the semi-lenses
(viz. horizontal triangles) on $\Omega$ and where two semi-lenses are connected by
edge in $\Gamma$ if they share an edge and are of the same type,
i.e., both are lower or both are upper ones. Let $\Phi$ be  a connected component of $\Gamma$ formed by \emph{upper} semi-lenses. Considering semi-lenses sharing edges, we can conclude that all semi-lenses in $\Phi$ have the same root $X\subset\mmmm$, i.e., each vertex occurring in a semi-lens there is of the form $Xi$ with $i\in\mmmm$ (see Sect.~\SSEC{combi}). Each  triangle $\tau\in \Phi$, having edges of types $ij,jk,ik$ for $i<j<k$ say, is the section $S_h(C)$ at
level $h=1$ of some cube $C=C(\tau)$ in $Q$; more precisely, $C$ is of the form $(X\,|\,T=\{i,j,k\})$, i.e., $C$ has the
lowest point at the vertex $X$ in $Q$, and edges of colors $i,j,k$. Then $\tau$ is the horizontal facet of the fragment $\Cfrag_1$ in $\Qfrag$ with the bottom vertex $X$ of height $|X|=m-1$ in $Z$.

The union of triangles $\tau\in\Phi$, denoted by $\frakU_\Phi$, forms a (possibly big) upper semi-lens. This is nothing else than an  upper semi-lens in the corresponding fine quasi-combi $\tilde K$ related to $K$. In its turn, the union of  fragments $\Cfrag_1$ over the cubes $C=C(\tau)$ for $\tau\in\Phi$ forms the convex truncated cone $D_\Phi$ with the bottom vertex $X$ and the upper (horizontal) side $\frakU_\Phi$.

Symmetrically, for each connected component $\Psi$ formed by \emph{lower} semi-lenses in $\Gamma$, all semi-lenses have the same ``upper'' root $Y$ (i.e., their vertices are expressed as $Y-i$), and each triangle $\sigma\in\Psi$ is the
section $S_2(C)$ at level 2 of some cube $C=C(\sigma)$ with the top vertex $Y$ (which is of height $|Y|=m+1$ in $Z$). Then the union $\frakL_\Psi:=\cup(\sigma\in\Psi)$ is a lower semi-lens as well (which is a tile in the fine quasi-combi $\tilde K$). And the union $F_\Psi$ of the upper fragments (simplexes) $\Cfrag_3(\sigma)$ over $\sigma\in\Psi$ forms the truncated cone with the top vertex $Y$ and the lower (horizontal) side $\frakL_\Psi$.

Note that the fine quasi-combi $\tilde K$ related to the ftq-combi $K$ is symmetric (since it is determined by the vertex set $V_K$, which is symmetric, by \refeq{QOmega}). This implies that the involution $\mu$ on $\Omega$ sends each upper semi-lens of $\tilde K$ to a lower one (and vice versa), which in turn implies that the corresponding lower and upper cones are symmetric to each other: $F_\Psi=\mu(D_\Phi)$.
The picture below illustrates such cones when  $|\Phi|=2$.

\vspace{-0.2cm}
\begin{center}
\includegraphics[scale=0.8]{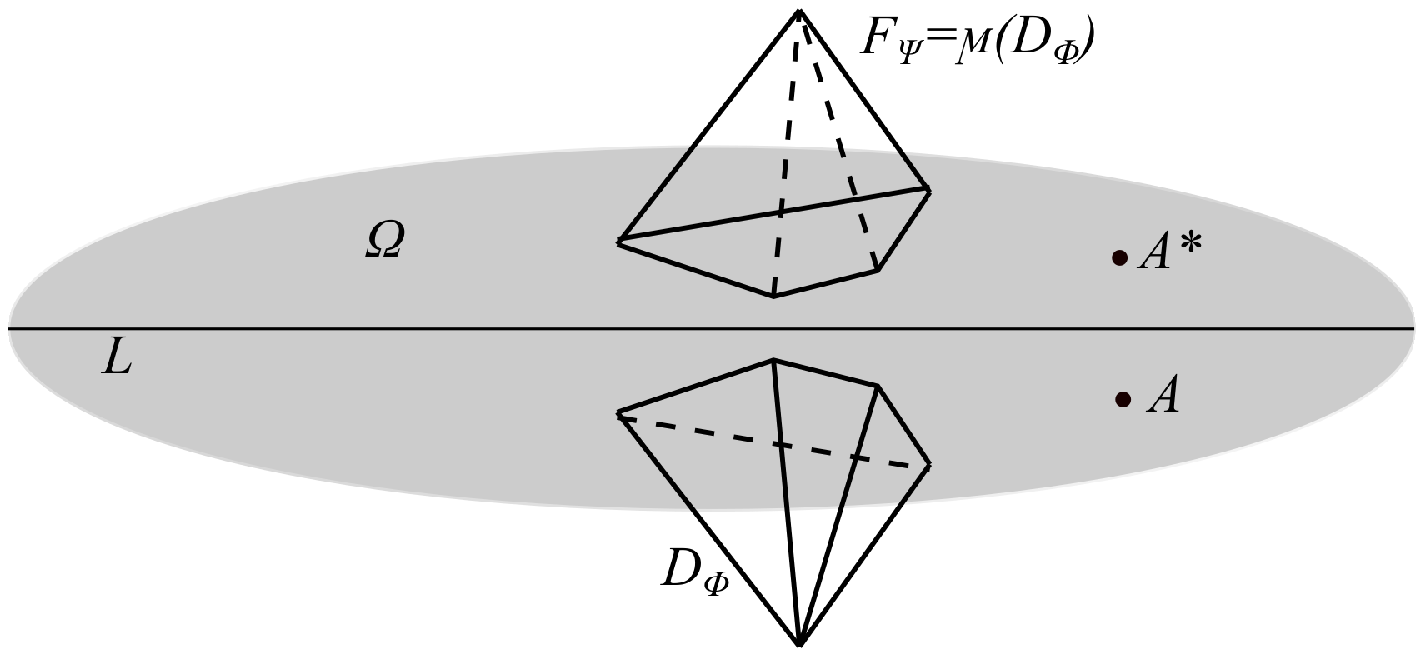}
\end{center}
\vspace{-0.5cm}

Using the above constructions, we perturb the plate $\Omega$ as follows. For each component $\Phi$ of upper type in $\Gamma$, the interior of the big semi-lens $\frakU_\Phi$ is replaced by the side surface of the cone $D_\Phi$ (as though squeezing out a ``pit'' in place of this semi-lens). Similarly, for each component $\Psi$ of lower type, the interior of  $\frakL_\Psi$ is replaced by the side surface of the cone $F_\Psi$ (making a ``peak''). Let $\tilde \Omega$ denote the resulting surface (which has the same boundary as in $\Omega$ but consists of a gathering of pits and peaks, without horizontal pieces at all).  Then the above-mentioned correspondence between the lower and upper cones implies that the perturbed $\tilde \Omega$ is symmetric w.r.t. the axis $L$. 

Now take the sub-cubillage $Q'$ of $Q$ formed by all cubes whose bottom vertex is of height at least $m-1$. Then $\tilde\Omega$ is exactly the lower boundary
of $Q'$. The symmetry of $\tilde\Omega$ implies that the cubillage $Q''$
symmetric to $Q'$ has $\tilde\Omega$ as the upper boundary. As a result, $Q'\cup Q''$ is a correct symmetric cubillage on $Z$ containing
$\Cscr$.

This contradicts the assumption that $\Cscr$ is maximal and $|\Cscr|<c_{2m}$, completing  the proof of Theorem~\ref{tm:symm-chord} for $n$ even.
\hfill\qed\qed
 \medskip

As a consequence of the above proof, any maximal c-collection $\Cscr$ in
$2^\mmmm$ is representable as the spectrum $V_Q$ of a symmetric cubillage $Q$ on $Z(2m,3)$ (which is determined by $\Cscr$, due to Theorem~\ref{tm:gal}). In particular, this implies that
  \begin{numitem1} \label{eq:sym-Omega}
the subcomplex (facet structure) $Q_\Omega$ of $\Qfrag$ on the plate $\Omega$ is
symmetric.
  \end{numitem1}
  
We finish this section with one more observation. As is shown in~\cite{DKK3} (and quoted in~\refeq{membr-combi}), for any maximal w-collection $\Wscr\subseteq 2^{[n]}$, there exist a cubillage $Q$ and a w-membrane $N$ in its fragmentation $\Qfrag$ such that $V_N=\Wscr$ (moreover, one shows there how to construct such $Q$ and $N$ efficiently). In particular,  $V_Q$ gives a maximal c-collection including $\Wscr$. A symmetric counterpart of such a relation between w- and c-collections is as follows.

\begin{theorem} \label{tm:symm-w-c}
For $n$ even, any maximal symmetric w-collection $\Wscr\subseteq 2^{[n]}$ can be represented as the spectrum (vertex set) of a symmetric w-membrane $N$ in the fragmentation of some symmetric cubillage $Q$ on $Z(n,3)$ (and such $Q$ and $N$ can be constructed efficiently). In particular, $V_Q$ is a maximal symmetric c-collection including $\Wscr$. Moreover, if $K$ is an arbitrary symmetric ftq-combi with $V_K=\Wscr$, then $Q$ and $N$ can be chosen so that $K=\pi^\rho(N)$.
  \end{theorem}
 \begin{proof}
For a maximal symmetric w-collection $\Wscr\subseteq 2^{[n]}$ with $n$ even, take a symmetric ftq-combi $K$ with $V_K=\Wscr$ (which exists and can be constructed as described in Case~1 of the proof of Theorem~\ref{tm:symm-ws} in Sect.~\SEC{symm-w-even}). This $K$ is the projection by $\pi^\rho$ of a w-membrane $N$ in some cubillage $Q'$ on $Z=Z(n,3)$; see~\refeq{membr-combi}. The symmetry of $K$ implies that $N$ (regarded as a complex) is stable under the involution $\mu$ on $Z$. This and the fact that the front side $\Zfr$ and the rear side $\Zrear$ of $Z$ are symmetric to each other (by $\mu$) imply that the part $Z'$ of $Z$ between $\Zfr$ and $N$ is symmetric to the part $Z''$ between $N$ and $\Zrear$. Now take the sub-complex $B$ of the fragmentation $\Qpfrag$ lying in $Z'$. Then $\mu(B)$ gives a subdivision of $Z''$ respecting $N$. In view of~\refeq{hor-sect}, any pair of fragments of $B$ and $\mu(B)$ sharing  a horizontal facet in $N$ must belong to the same cube. Using this, one can conclude that  the union of $B$ and $\mu(B)$ constitutes a well-defined fragmentation of a symmetric cubillage $Q$ on $Z$ containing $N$, whence the theorem follows. 
 \end{proof}


\section{Maximal symmetric c-collections: odd case} \label{sec:symm-c-odd}

In this section we prove Theorem~\ref{tm:symm-chord} when the number $n$ of
colors is odd, $n=2m+1$. We assume that the set of colors is given in the
symmetrized form, namely, as $\mmm$ ($=\{-m,\ldots,-1,0,1,\ldots,m\}$). Our
goal is to show that any symmetric chord separated collection $\Dscr\subset
2^\mmm$ is included in the spectrum $V_Q$ of a symmetric cubillage $Q$ on the symmetric zonotope $Z(n,3)$, whence the theorem follows from Galashin's
results (\refeq{gal} and Theorem~\ref{tm:gal}). We will use a reduction to the
even case as in Sect.~\SEC{symm-c-even}.

First of all we partition $\Dscr$ as $\Dscr'\sqcup\Dscr''$, where
   $$
   \Dscr':=\{A\in\Dscr\colon 0\notin A\}\quad \mbox{and} \quad
   \Dscr'':=\{A\in\Dscr\colon 0\in A\}.
   $$
By the symmetry of $\Dscr$, the map $A\mapsto A^\ast 0$ gives a bijection between
$\Dscr'$ and $\Dscr''$.

Define the collections $\Cscr'$, $\Cscr''$, $\Cscr$ in $2^\mmmm$ (where, as
before, $\mmmm$ means $\mmm-\{0\}$) by
  $$
  \Cscr':=\Dscr',\quad \Cscr'':=\{A-0\colon A\in\Dscr''\},\quad \mbox{and}
  \quad \Cscr:=\Cscr'\cup\Cscr''.
  $$

In what follows, we will use symbol $\ast$ for the K-involution in $2^\mmm$,
and $\natural$ for that in $2^\mmmm$, and similarly for the geometric versions of these involutions. The symmetry of $\Dscr$ implies that of $\Cscr$. Also it is not difficult to see that
  \begin{numitem1} \label{eq:C-csep}
$\Cscr$ is chord separated;
  \end{numitem1}
we leave this, as well as verifications of assertions~\refeq{expanQ}
and~\refeq{E-csep} below, to the reader as an exercise.

As is shown in the proof of Theorem~\ref{tm:symm-chord} for the even case in the previous section,
$\Cscr$ is extendable to the spectrum $V_Q$ of a symmetric cubillage $Q$ on the
zonotope $Z_-:=Z(\Theta)\simeq Z(2m,3)$ (where, as before, $\Theta$ is a
symmetric set of generators $\theta_i$, $i\in\mmmm$, as
in~\refeq{Theta},\refeq{sym-theta}). We are going to expand $Q$ to a cubillage
$\tilde Q$ on the zonotope $Z\simeq Z(2m+1,3)$ generated by $\Theta$ plus the
vector $\theta_0$ which, w.l.o.g., can be defined as the vertical vector
$(0,1,0)$. Based on properties of cubillages (see~\cite[Sect.~3.3]{DKK3}), an expansion of
a cubillage to a larger one with an additional color is performed by use of a
certain strong membrane.

More precisely, in our case, by a strong membrane related to color 0, or a
\emph{0-membrane}, we mean a (closed two-dimensional) subcomplex $N$ of $Q$ such that
the projection $\pi_0:\Rset^3\to\Rset^2$ parallel to the vector $\theta_0$ is
injective on $N$ (regarded as a set of points), and $\pi_0(N)=\pi_0(Z_-)$. It divides $Z_-$ into two halves
$\Zlow_-(N)$ and $\Zup_-(N)$ consisting of the points lying below and above $N$
(including $N$ itself), respectively. Two additional conditions on $N$ are
needed to us, namely:
  \begin{numitem1} \label{eq:0membr}
 \begin{itemize}
\item[(i)] $N$ is symmetric, i.e., $N^\natural=N$; and
\item[(ii)] the vertices of $\Cscr'$ lie in $\Zlow_-(N)$, while those of
$\Cscr''$ in $\Zup_-(N)$;
  \end{itemize}
  \end{numitem1}
such an $N$ is called \emph{feasible}.

The \emph{expansion operation} on $Q$ using $N$ acts as follows: it splits $Q$
into two closed parts (subcomplexes) $Q'$ and $Q''$ lying in $\Zlow_-(N)$ and
$\Zup_-(N)$, respectively; the part $Q'$ is preserved, while $Q''$ moves by
$\theta_0$, and the gap between $Q'$ and $Q''+\theta_0$ is filled by new cubes,
each being the Minkowski sum of a facet (``rhombus'') of $Q$ contained in $N$ and the
segment $[0,\theta_0]$. The resulting complex is called the \emph{expansion} of
$Q$ by color 0 using $N$. One easily shows that
  \begin{numitem1} \label{eq:expanQ}
if $Q$ is symmetric and $N$ is feasible, then the expansion of $Q$ by color 0 using $N$ is a symmetric cubillage on $Z(2m+1,3)$ whose spectrum includes $\Dscr$.
  \end{numitem1}

In light of~\refeq{expanQ}, we will attempt to show that a feasible membrane in
an appropriate $Q$ does exist (whence the result will immediately follow). 

Our programme to accomplish this task is a bit
tricky. We first lift $\Dscr$ to a symmetric collection $\Escr$ in $2^\mmmp$,
where $\mmmp$ is $\mmmm$ to which two symmetric colors $0'$ and $0''$ are
added, namely, $\{-m,\ldots,-1,0',0'',1,\ldots,m\}$. Second, we include $\Escr$
in the spectrum of a symmetric cubillage  $R$ on the zonotope $Z_+\simeq
Z(2m+2,3)$ generated by $\Theta$ plus two vectors $\theta_{0'},\theta_{0''}$
viewed as $\theta_{0'}:=(-\vareps,1,0)$ and $\theta_{0''}:=(\vareps,1,0)$ for a
small real $\vareps>0$. Third, we reduce $R$ by eliminating colors $0'$ and
$0''$ so as to obtain a symmetric cubillage $Q$ on $Z_-\simeq(2m,3)$ equipped
with two particular 0-membranes $N'$ and $N''$ which are symmetric to each
other. Fourth, updating $N'$ and $N''$ symmetrically, step by step, we
eventually turn these into the desired symmetric membrane $N$ in $Q$.

Next we describe these stages more carefully, using symbol $\bullet$ to denote
the symmetry in $2^\mmmp$ (and the K-involution in $Z_+$). One can see that for
$X\subseteq\mmmp$, if $X$ contains none of $0',0''$, then
$X^\bullet=X^\natural\, 0'0''$; if $X$ contains $0'$ but not $0''$ (resp. $0''$
but not $0'$), then $X^\bullet$ is $(X-0')^\natural\, 0'$ (resp.
$(X-0'')^\natural\, 0''$); and if $X$ contains both $0',0''$, then $X^\bullet=(X-\{0',0''\})^\natural$.

Define the collections $\Escr'$, $\Escr''$, $\Escr$ in $2^\mmmp$ by
  \begin{equation} \label{eq:EEE}
  \Escr':=\Dscr',\quad \Escr'':=\{(X-0)\,0'0''\colon X\in\Dscr''\},
  \quad \mbox{and} \quad \Escr:=\Escr'\cup \Escr''.
  \end{equation}

The symmetry of $\Dscr$ implies that of $\Escr$, and a routine verification
shows that
  \begin{numitem1} \label{eq:E-csep}
$\Escr$ is chord separated.
  \end{numitem1}

In view of~\refeq{E-csep} and since $2m+2$ is even, there exists a symmetric
cubillage $R$ on $Z_+\simeq Z(2m+2,3)$ such that $\Escr\subset V_R$. Let
$\Pscr'$ ($\Pscr''$) be the set of cubes $C=(X\,|\,T)$ whose type $T$ contains
color $0'$ (resp. $0''$); following terminology in~\cite[Sect.~3.3]{DKK3}, we refer to this set (or its induced subcomplex in $R$) as the $0'$-\emph{pie} (resp.
$0''$-\emph{pie}) in $R$. By a simple, but important, property of pies,
$\Pscr'$ is representable as the (Minkowski) sum of a
$0'$-membrane $N'$ and the segment $[0,\theta_{0'}]$. Similarly, $\Pscr''$ is
the sum of a $0''$-membrane $N''$ and $[0,\theta_{0''}]$. 

One can check that for a cube $C=F+[0,\theta_{0'}]$ in $\Pscr'$, where $F$ is a facet in $N'$, its symmetric cube $C^\bullet$ in $R$ is viewed as $(F0')^\bullet+[0,\theta_{0''}]$, and therefore, $(F0')^\bullet$ belongs to $N''$, and $C^\bullet$ to $\Pscr''$. And similarly for cubes in $\Pscr''$. This and the identity $((F0')^\bullet 0'')^\bullet=F$ (where $F$ is in $N'$) imply that
  \begin{numitem1} \label{eq:piesPP}
the pies $\Pscr'$ and $\Pscr''$ are symmetric to each other, the
$0'$-membrane $N'$ is symmetric to the $0''$-membrane $N''+\theta_{0''}$, and similarly, $N''$ is symmetric to $N'+\theta_{0'}$.
  \end{numitem1}

Now consider the collections $\Escr'$ and $\Escr''$ as in~\refeq{EEE}. Since
the sets in $\Escr'$ (in $\Escr''$) do not contain colors $0',0''$ (resp.
contain both $0',0''$), the vertices of $R$ representing these sets are located (non-strictly) below (resp. above) $\Pscr'\cup\Pscr''$. We further need to use the
contraction operations on pies of cubillages, which are converse, in a sense,
to the expansion operations on membranes; for details, see~\cite[Sect.~3.3]{DKK3}.

The \emph{contraction operation} applied to $\Pscr'$ removes from $R$ the cubes
lying between the membranes $N'$ and $N'+\theta_{0'}$ and moves the part of $R$
above $N'+\theta_{0'}$ by $-\theta_{0'}$, so that the images of these membranes become glued together. The contraction operation applied to $\Pscr''$ (or to the image of $\Pscr''$ after contracting $\Pscr'$) acts similarly. These two operations
commute, and applying both (in any order), we obtain a symmetric cubillage $Q$
on $Z_-$ and two membranes $\tilde N'$ and $\tilde N''$ in it, which are the
images of $N',N''$ and can be simultaneously thought of as 0-membranes since the vectors
$\theta_0,\theta_{0'},\theta_{0''}$ are close to each other. Moreover, since the vertices in $\Escr'$ contain none of colors $0',0''$ (and therefore lie below the pies $\Pscr',\Pscr''$), whereas those in $\Escr''$ contain both $0',0''$ (and lie above these pies), we observe, by comparing $\Cscr'$ with $\Escr'$, and $\Cscr''$ with $\Escr''$, that
  \begin{numitem1} \label{eq:locCC}
$\Cscr'$ lies in the region $\Zlow_-(\tilde N')\cap \Zlow_-(\tilde N'')$ (i.e.,
below both $\tilde N',\tilde N''$), while $\Cscr''$ lies in $\Zup_-(\tilde
N')\cap \Zup_-(\tilde N'')$ (i.e., above both $\tilde N',\tilde N''$).
  \end{numitem1}

Next, for a closed subset $U$ of points in $Z_-$, define $\Ufr_0$ (resp. $\Urear_0$) to be
the set of points $u\in U$ ``seen from below'' (resp.``seen from above''), i.e., such that $U$ has no point of the form
$u-\delta\theta_0$ (resp. $u+\delta\theta_0$) with $\delta>0$; we call this the
\emph{0-front} (resp. \emph{0-rear}) \emph{side} of $U$. It is not difficult to show (cf.~\cite[Sect.~3.4]{DKK3}) that any set $S$ of cubes in a cubillage on $Z_-$ has a minimal (maximal) element $C$, in the sense that there is no $C'\in S$ such that $\Cfr_0$ and $\Cprear_0$ (resp. $\Crear_0$ and $\Cpfr_0$) share a facet. (In other words, there is no cube in $S$ before (resp. after) $C$ in the direction $\theta_0$.)

Define
  $$
  H':=(\tilde N'\cup \tilde N'')^{\rm fr}_0\quad \mbox{and}
  \quad H'':=(\tilde N'\cup \tilde N'')^{\rm rear}_0.
  $$

One can see that $H',H''$ are 0-membranes in $Z_-$, $H''$ is symmetric to $H'$
and lies (non-strictly) above $H'$, and in view of~\refeq{locCC}, the
collection $\Cscr'$ lies below $H'$, while $\Cscr''$ does above $H''$. Now the
result is provided by the following
  \begin{lemma} \label{eq:between}
There exists a symmetric 0-membrane $H$ lying between $H'$ and $H''$.
  \end{lemma}
  \begin{proof}
If $H'=H''$, we are done. So assume that the set $\Sscr$ of cubes of $Q$
filling the gap $\omega$ between $H'$ and $H''$ is nonempty. 

Let $C$ be a maximal cube in $\Sscr$ (in the sense explained above). Then the 0-rear side $\Crear_0$ of $C$ is entirely contained in $H''$ (for if a facet $F$ of $\Crear_0$ is not in $H''$, then there is a cube $C'$ with $F\subset \Cpfr_0$, and this $C'$ lies in $\omega$ and is greater than $C$). It follows that replacing in $H''$ the side (disk)
$\Crear_0$ by $\Cfr_0$, we obtain a correct 0-membrane $\hat H''$ lying in
$\omega$. By the symmetry of $\omega$, the cube $C^\natural$ symmetric to $C$
lies in $\omega$ as well, and its 0-front side is entirely contained in $H'$.
Let $\hat H'$ be the 0-membrane obtained by replacing in $H'$ the 0-front side
of $C^\natural$ by its 0-rear side. 

An important fact is that the cubes $C$ and
$C^\natural$ are different. This is so because if $C$ has type $\{i,j,k\}$ (the edge colors in $C$), then, by symmetry, $C^\natural$ has type $\{i^\circ,j^\circ,k^\circ\}$, and $\{i,j,k\}=\{i^\circ,j^\circ,k^\circ\}$ is impossible since $\mmmm$ has no
color symmetric to itself. The transformation $(H',H'')\mapsto (\hat H', \hat H'')$ is illustrated in the picture (where for simplicity we give two-dimensional projections).

\vspace{-0.2cm}
\begin{center}
\includegraphics[scale=0.8]{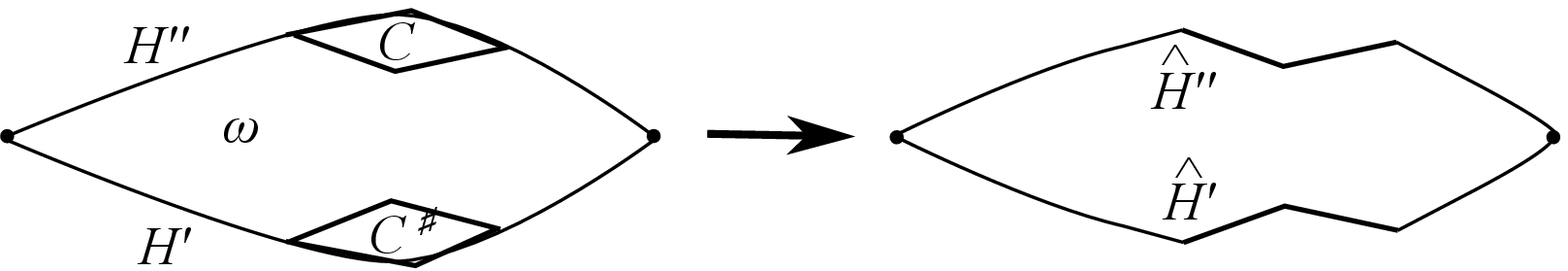}
\end{center}
\vspace{-0.2cm}

So the new 0-membranes $\hat H'$ and $\hat H''$ are symmetric to each other,
the former lies (non-strictly) below the latter, and the gap between them is
smaller than $\omega$. Repeating the procedure, step by step, we eventually
obtain a pair for which the gap vanishes, yielding the desired $H$, as required.
  \end{proof}

This completes the proof of Theorem~\ref{tm:symm-chord} for the case of odd colors. \hfill\qed\qed
  \medskip

\noindent\textbf{Remark 6.} For an integer $n>0$, let us say that $J\subseteq
[0..n]$ is symmetric if $k\in J$ implies $n-k\in J$. Let $\Lambda_n(J)$ be the
collection of sets $\binom{[n]}{k}$ over $k\in J$. From Theorem~\ref{tm:symm-chord} one can obtain a slightly sharper purity result, namely:
 \begin{numitem1} \label{eq:sharper}
for a fixed symmetric $J\subseteq[0..n]$, all inclusion-wise maximal symmetric
chord separated collections in $\Lambda_n(J)$ have the same size.
  \end{numitem1}

Indeed, consider such a collection $\Cscr$ and extend it to a maximal symmetric
c-collection $\Sscr$ in $2^{[n]}$. From the equality $|\Sscr|=c_n$ (by
Theorem~\ref{tm:symm-chord}) in follows that for each $k\in[0..n]$, the number
of sets $A\in \Sscr$ with $|A|=k$ is equal to exactly $k(n-k)+1$. Also the
symmetry of $J$ implies that for each $k\in J$, the subcollection
$\Cscr\cap\binom{[n]}{k}$ coincides with $\Sscr\cap\binom{[n]}{k}$, and the
result follows.
 \medskip
 
 In conclusion of this paper, we give an analog of Theorem~\ref{tm:symm-w-c} for $n$ odd. Recall that in this case the maximal size of a symmetric w-collection $\Wscr$ in $2^{[n]}$ is $(n-1)/2$ units less than the maximal size $s_n$ for the non-symmetric case. For this reason, when dealing with corresponding cubillages $Q$ on $Z(n,3)$, we need to consider a slightly wider class of membranes. Namely, compared with the definition of a weak membrane given in Sect.~\SEC{symm-c-even} where merely horizontal and vertical  facets in the fragmentation $\Qfrag$ are allowed to use, we add one more type of facets $F$. Such an $F$ is obtained as the \emph{middle vertical section} of the octahedron $\Cfrag_2$ in the fragmentation of a cube $C=(X\,|\, T)$ with colors $T$ of the form $(i<0<i^\circ)$, i.e. $F$ is the rectangle spanning the vertices $Xi,Xi0,Xi^\circ,X0i^\circ$. In addition, we assume that such facets can appear only in the middle level of $Q$, namely, when $|X|=m-1$, where $n=2m+1$.We refer to the corresponding subcomplexes in the \emph{extended fragmentation} of $Q$ constructed in this way (and respecting the projection $\pi^\rho$) as \emph{extraordinary weak membranes}.

\begin{theorem} \label{tm:symm-w-c-odd}
For $n$ odd, any maximal symmetric w-collection $\Wscr\subseteq 2^{[n]}$ can be represented as the spectrum of a symmetric extraordinary weak membrane $N$ in the extended fragmentation of some symmetric cubillage $Q$ on $Z(n,3)$ (and such $Q$ and $N$ can be constructed efficiently). In particular, $V_Q$ is a maximal symmetric c-collection including $\Wscr$.
  \end{theorem}
  
\begin{proof}
For $n=2m+1$, let $\Wscr$ be a symmetric w-collection in $2^{[n]}$ of size $s_n-m$. As is explained in Sect.~\SEC{symm-w-odd}, in the symmetric zonogon $Z$ generated by vectors $\xi_{-m},\ldots\xi_{-1},\xi_0,\xi_{1}\ldots,\xi_m$ as in~\refeq{symgen}, one can construct a tiling  $K$ on $Z$, called a \emph{pseudo-combi}, such that: (a) $V_K=\Wscr$; (b) the sub-tiling $\Klow$ of $K$ below the horizontal line $M$ (cf.~\refeq{M}) is symmetric to the sub-tiling $\Kup$ of $K$ above the line $M^+=M+\xi_0$; (c) combining $\Klow$ with $\Kup$ shifted by $-\xi_0$ results in a symmetric ftq-combi on the reduced $2m$-colored zonogon; and (d) the rectangular gap $F$ between $M$ and $M^+$ in $Z$ is subdivided into rectangles $F_1,\ldots,F_m$. 

This $K$ can be transformed into a correct (non-symmetric) ftq-combi $\tilde K$ by subdividing each $F_p$ into the quadruple $\Tscr_p$ of triangles as is~\refeq{rect-4triang} (see also the picture above Remark~4); so the non-symmetry of $\tilde K$ concerns merely the part $F$. Take a cubillage $\tilde Q$ on $Z'\simeq Z(n,3)$ and a w-membrane $\tilde N$ in $\tilde Q$ such that $\tilde K=\pi^\rho(\tilde N)$ (existing by~\refeq{membr-combi}). One can realize that each quadruple $\Tscr_p$ in $\tilde K$ corresponds, in the fragmentation ${\tilde Q}^\equiv$ of $\tilde Q$, to the rear side $D(p)^{\rm rear}$ of the octahedral (middle) fragment $D(p)$ of  some cube $C(p)$ of $\tilde Q$, i.e., $\pi^\rho$ maps the facets of $D(p)$ to the triangles of $\Tscr_p$.  (This cube has colors $i<0<i^\circ$, where $F_p$ is of the form $e_p\times \xi_0$ and the horizontal edge $e_p$ has type $ii^\circ$.) Then the middle vertical section $S_p$ of $C(p)_2^\equiv$  corresponds to the rectangle $F_p$ in $K$, namely, $\pi^\rho(S_p)=F_p$.

Now for each $p=1,\ldots,m$, replace in $\tilde N$ the part $D(p)^{\rm rear}$ by $S_p$. This gives an extraordinary weak membrane $N$ which is bijective (under $\pi^\rho$) to the pseudo-combi $K$. In the extended fragmentation of $Q'$, take the sub-complex $G$ lying between $\Zpfr$ and $N$, and form its symmetric complex $G'$. The facts that $N$ is self-symmetric and $\Zpfr$ is symmetric to $\Zprear$ imply that $G'$ is a subdivision of the part of $Z'$ between $N$ and $\Zprear$. Then the union of $G$ and $G'$ contains $N$ (as a sub-complex) and constitutes the correct extended fragmentation of  a symmetric cubillage $Q$ (taking into account that for each section $S(p)$, the fragments in $G$ and $G'$ sharing $S(P)$ are symmetric to each other and their union forms the octahedral  fragment of a symmetric cube). This implies the theorem.
\end{proof}

\noindent\textbf{Remark 7.} A similar result is valid for the strong separation, namely: for $n$ odd, any maximal symmetric s-collection $\Sscr\subseteq 2^{[n]}$ can be represented as the spectrum of a symmetric extraordinary strong membrane $N$ in the sparse fragmentation of some symmetric cubillage $Q$ on $Z(n,3)$. Here the \emph{sparse fragmentation} concerns transformations of merely self-symmetric cubes $C=(X\,|\,T)$, i.e., such that $|X|=m-1$ (where $n=2m+1$) and $T=(i<0<i^\circ)$ for $i\in[-m..-1]$, under which $C$ is subdivided into two parts (symmetric to each other) sharing the vertical section (rectangle) $S(C)$ spanning the vertices $Xi,Xi^\circ,Xi0,X0i^\circ$. Accordingly, an \emph{extraordinary strong membrane} differs from a usual strong membrane by allowing to use as facets such sections $S(C)$ (and the other facets are 
halves (vertical triangles) of facets of cubes, as usual). To obtain the assertion, we take the symmetric non-standard tiling $T$ with $V_T=\Sscr$ on $Z(n,2)$ as described in Remark~4. In the construction of the desired $Q$ and $N$, each central hexagon $H$ in $T$ should be the projection (by $\pi^\rho$) of the union of three facets in the sparse fragmentation of a self-symmetric cube $C$ in $Q$ (of which one is the rectangular section $S(C)$). We leave details to the reader.


\end{document}